\documentclass[12pt,english]{article}
\usepackage{mathptmx}
\usepackage{geometry}
\usepackage{xcolor}
\usepackage{array}
\usepackage{bbding}
\usepackage{multirow}
\usepackage[fleqn]{amsmath}
\usepackage{amssymb}
\usepackage{amsthm}
\usepackage{graphicx}
\usepackage{natbib}
\usepackage{lscape}
\usepackage[hyphens]{url}
\usepackage[hidelinks]{hyperref}
\usepackage{comment}
\usepackage{bm}
\usepackage[title]{appendix}
\usepackage{longtable}
\usepackage{mathtools}
\usepackage{chngcntr}
\usepackage{authblk}
\usepackage{babel}
\usepackage{dsfont}
\usepackage{subcaption}
\usepackage{fancyhdr}
\usepackage{cleveref}
\usepackage{abstract}

\date{}

\usepackage[mathlines, pagewise]{lineno}
\usepackage{etoolbox} 

\newtheorem{theorem}{Theorem}

\usepackage{tikz}
\usepackage{enumerate}
\usepackage{graphicx}
\usepackage{subcaption}
\usepackage{verbatim}
\usepackage{comment}
\usetikzlibrary{decorations.pathreplacing, arrows.meta}

\usepackage[textsize=tiny]{todonotes}

\newtheorem{lemma}{Lemma}
\newtheorem{corollary}{Corollary}
\title{The Increasing Gap Dynamics in a General Spatial Matching Model}

\newcommand{\PP}{\mathbb{P}}
\author[a,$\ast$]{Andres Fielbaum}
\author[b,c]{Roberto Cominetti}
\author[d]{Jose Correa}
\affil[a]{School of Civil Engineering, University of Sydney, Darlington 2008, NSW, Australia}
\affil[b]{
Institute for Mathematical and Computational Engineering, Faculty of  Mathematics and School of  Engineering, Pontificia Universidad Cat\'olica, Santiago, Chile}
\affil[c]{Departament of Industrial and Systems Engineering, Faculty of  Engineering, Pontificia Universidad Cat\'olica, Santiago, Chile}
\affil[d]{Department of Industrial Engineering, Universidad de Chile, Santiago, Chile}

\begin{document}
\maketitle

\noindent{\bf Abstract:} We study a representation of a problem that appears in numerous transport systems: $N$ servers distributed over a given space (e.g., cars on an urban network), receive random requests from arriving users who get 
assigned to the closest server, after which this server is replaced by a new one at a random location. We show that this creates a negative feedback loop, which we call \textit{Increasing Gap Dynamics} (IGD): when a server is assigned a spatial gap forms, which is more likely to attract new users that further widen the gap. 

The simplest version of our model is a one-dimensional circle, for which we derive analytical results showing that the system converges to an inefficient equilibrium, worse than both balanced and fully random distributions of servers. We prove that an optimal assignment policy always matches the user to one of its two neighbouring servers so that long gaps tend to widen. Hence, the IGD persists even when assigning optimally rather than greedily. In two dimensions, the appearance of the IGD is illustrated through simulations on a square region. Finally, simulations of a proper ride-hailing system using real data from Manhattan confirms that the IGD arises and that it is responsible for the appearance of the well-known Wild Goose Chase.

\vspace{3ex}



\section{Introduction: A Negative Feedback in Spatial Matching}\label{sec:INTRO}
Consider the following generic situation. There is a spatial domain, like a graph or an Euclidean region, and users who appear randomly in space at a roughly constant rate. These users must be matched with a server; the closer, the better.  Concurrently, servers also appear randomly in space after finishing service for a previous user, potentially at a different location from where they started. This informal and abstract formulation can represent various transport systems, such as: 

\begin{itemize}
    \item Ride-hailing: Servers are drivers who appear upon completion of a ride and become available again.
    \item Dockless bike-sharing: Servers are the bicycles.
    \item On-street parking when delivering: Users are the households where items must be delivered, and servers are parking lots where the deliverer will park.
\end{itemize}

What is the evolution of these types of systems? Can we have a general model representing all of them? In particular, in all the examples above, it has been found that situations like the \textit{Wild Goose Chase} (WGC) can occur \citep{xu2024economic,castillo2017surge,ouyang2023measurement,emami2024integrated,li2025wild,fielbaum2025coordination}. The WGC refers to a situation where the spatial distance between matched users and servers becomes inefficiently large. The WGC is undesirable but can be an equilibrium because the long distance implies that servers remain occupied for long periods, as the server is busy since the moment the users are matched, thereby reducing the number of available servers. We refer to this negative feedback cycle between the number of idle servers and serving time as the \textit{WGC cycle}. While this bad equilibrium has been clearly reported in the literature, the dynamics that explain why WGC appears in the first place have remained underexplored.

In this paper, we study the spatial evolution of this generic system. We propose a simple Markovian model and provide several theoretical and experimental results, as well as a few open questions. We also explain the mechanisms that make the WGC to appear. The reason is actually quite simple and rests on a fundamental limitation, namely the impossibility of keeping a balanced distribution of the servers (unless the system can be intervened and thus the servers rebalanced). Let us first explain it intuitively using Fig. \ref{fig:ExampleIntro}:

\begin{enumerate}[(a)] 
    \item  A circle with servers (cars) uniformly distributed, where we highlight the \textit{dominance zone} of every server $s$. That is, the region where $s$ is the nearest compared to others so that a user appearing there would be assigned to $s$ assuming a greedy rule.
    \item A user appears, and it is assigned to the closest car.
    \item The car is now taken so the two surrounding dominance zones are merged into a larger zone.
\end{enumerate}

\begin{figure}[ht]
    \centering
    \includegraphics[width=0.8\linewidth]{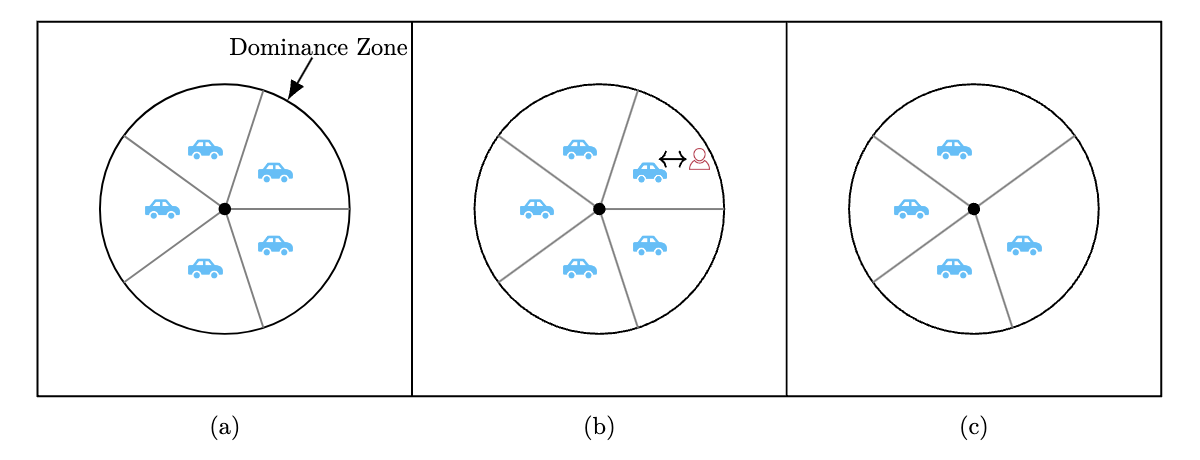}
    \caption{An example of the dynamic that prevents servers from remaining balanced in a spatial setting.}
    \label{fig:ExampleIntro}
\end{figure}

Even when a new car appears, the size of the dominance zones will now be heterogeneous. The larger areas will have a greater probability of receiving the next user, and when this happens, there are two negative effects: 

\begin{enumerate}
    \item The user will, on average, face an increased cost, as measured by the distance to its closest server.
    \item The large area will increase even further, making the problem worse for future users. That is, there is \textbf{a negative feedback loop}, where the gaps created when one server disappears (gets assigned) tend to become bigger and bigger. We call this the \textit{Increasing Gaps Dynamics} (IGD).
\end{enumerate}

The IGD keeps worsening the servers' distribution until a probabilistic equilibrium is reached. As will be shown in the subsequent sections, this equilibrium is not only worse than a uniform distribution of servers but also worse than a fully random one.

We study and demonstrate these dynamics using three different techniques. First, we focus on what could be considered to be the simplest possible setting, namely a one dimensional circle, where we provide a number of theoretical results. We then consider a unit square in two dimensions, where analytical results seem beyond reach but where we present visualisations to further develop the intuition. Finally, we include a set of results of microsimulations using a real-world dataset from Manhattan (New York),  showing that these dynamics also hold under realistic configurations. To enhance readability, the rest of this paper assumes we are describing ride-hailing, but the model remains general.

The rest of the paper is structured as follows. Section \ref{sec:LitReview} gives an overview of previous studies related to this one. Section \ref{sec:Model} introduces the formal mathematical model. Section \ref{sec:1d} provides a detailed theoretical analysis of the dynamics on the one-dimensional circle. Section \ref{sec:Square} focuses on the unit square and provides visualisations and a video to strengthen the understanding of these dynamics. Section \ref{sec:Simulations} shows the results of several simulations on a regular grid and on Manhattan. Section \ref{Sec:SummaryDynamics} consolidates the two relevant dynamics that take place in these systems: IGD and WGC, through a schematic explanation of the feedback loops involved. Finally, section \ref{sec:Conclusions} concludes and suggests directions for further research.

\section{Related works}\label{sec:LitReview}
Our paper describes the evolution of dynamic spatial matching models. Therefore, two lines of research are the most relevant to review. First, the challenges of making decisions with partial information in these systems, and second, their analysis at equilibrium.

\subsection{Forward-looking decisions in on-demand mobility}
App-based on-demand mobility has become very popular in recent years. Different from traditional on-demand modes (e.g. taxis), the apps have the ability to collect in real time the information of all drivers and users, to then decide which is the best assignment.\footnote{Note that drivers are often independent decision-makers \citep{ashkrof2022ride,ashkrof2024relocation}. Nevertheless, most of the papers reviewed here assume that the platform is able to decide centrally.} However, the information they collect is only about the current users and drivers; and the optimal decision contingent on that information might differ from what would be decided if all the information were known beforehand \citep{berbeglia_dynamic_2010}.

Researchers have proposed various ways to deal with this challenge. First, there are several methods to decide how to \textit{rebalance} the fleet, i.e., to instruct the idle vehicles to relocate at places that are expected to be undersupplied. \cite{guo2021robust,yang2022learning} propose learning-based methods for rebalancing in ride-hailing systems, while \cite{pan2019deep} also develops machine learning to rebalance dockless bike-sharing, a type of system where empirical studies show that only some stations can achieve ``self-balancing'' \citep{hua2025can}.

An interesting difference emerges when several passengers can share the vehicle at the same time. In this case, it might be optimal to follow a route that is not the shortest path, to try to find more passengers along the way; additionally, the decision of whom to pick up and drop off first can also include anticipatory techniques. These ideas are the focus of the papers by \cite{alonso2017predictive,fielbaum2022anticipatory}.

Instead of instructing the vehicles to move in a certain way, similar behaviours can be induced through pricing. In ride-hailing, this is known as \textit{surge pricing}, namely, when zones that are undersupplied face an increase in their fares in order to attract more drivers and deter some users. The dynamics and equilibria have been studied by \cite{besbes2021surge,hu2022surge,yan2020dynamic}. Similar ideas have been explored in car-sharing \citep{zakharenko2023pricing} and bike-sharing \citep{emami2024integrated}.

The uncertainty about the future state of the system not only affects its efficiency: in these systems, \textit{unreliability} is often a critical issue, as users cannot know in advance the characteristics of their assignment. In the case of ride-hailing, \cite{bansal2020impact} estimate that the demand could increase by 10\% if waiting times could be predicted accurately; in ride-pooling, the routes can adapt dynamically to accommodate to the new requests, and this factor deters a more widespread adoption of these systems \citep{fielbaum2020unreliability,zhang2024barrier,alonso2020value}; regarding bike-sharing, in dockless systems the time until finding an available bicycle can be significantly less predictable than when systems are station-based \citep{beauvoir2020unreliability}. 

\subsection{Spatial equilibrium models}
A different stream of research, like ours, proposes simplified models that allow for analytical treatment. Many of these works focus on identifying the types of equilibria that can arise. As a result, the models are typically \textit{static}, describing what the equilibrium looks like rather than how it is reached.

The most relevant model in this context was developed by \cite{castillo2017surge}, who model the equilibrium between supply and demand. This framework was further extended by \cite{yan2020dynamic} and \cite{fielbaum2024idle}. This model includes the analysis of the relationship between the number of idle vehicles and the waiting time in ride-hailing, but can be easily applied to other similar modes, and it explains why the Wild Goose Chase can be an equilibrium. In this line of models, the spatial dimension is not considered explicitly, or equivalently, the system is assumed to be homogeneous in space.

Other papers have developed explicit physical models to analyse either the matching or the pickup times in the case of ride-hailing \citep{zhang2019efficiency,zhang2025walking,zha2016economic,yang2020optimizing,li2024aggregate}. These papers have proved useful in developing a theoretical understanding of several aspects, such as scale economies, comparison with traditional taxis, passengers' competition for vehicles, and the role of walking.
Finally, the papers by \cite{kanoria2021dynamic,balkanski2023power} are closer to ours as they develop mathematical models for spatial matching. Details are discussed later when introducing our model.

\section{A General Model}\label{sec:Model}

Let $(X,d)$ be a metric space (i.e., a set $X$ endowed with a distance function $d$), and $S^1=(s^1_1,\ldots,s^1_N) \in X^N$ a tuple representing the initial location of $N$ vehicles in $X$. Consider two random sequences $(x_t)_{t\geq 1}$ and $ (y_t)_{t\geq 1}$ in $X$, representing the arrivals of users and new vehicles, with   i.i.d. distributions $Q$ and $P$ respectively. The process $(S^t)_{t\ge 1}$ goes as follows. At every stage $t \in \mathbb{N}$:
\begin{enumerate}
    \item A newly arrived user at $x_t\sim Q$ is \textit{assigned} to the closest vehicle, i.e., to $k^*=\text{argmin}_{k=1,\ldots,N}~d(x_t,s^t_k)$, facing a \textit{cost} of $c_t=d(x_t,s^t_{k^*})$. This implies the assignment is done greedily, an assumption that will be discussed later in section \ref{sec:Greedy?}.
    \item The assigned vehicle at $s^t_{k^*}$ is no longer idle and is replaced by a newly arrived vehicle at the random location $y_t\sim P$, so that 
     the new vector of vehicle positions is $S^{t+1}=(s^t_1,\ldots,s^t_{k^*-1}, y_t,s^t_{k^*+1},\ldots,s^t_N)$.
\end{enumerate}

\noindent Typical examples of $X$ are the nodes of a graph with $d$ the length of the shortest paths, or a subset of $\mathbb{R}$ or $\mathbb{R}^2$ (as in Fig. \ref{fig:ExampleIntro}) using Euclidean distance. The case of $X=[0,1]^d$ with Euclidean distance has been previously proposed by \cite{kanoria2021dynamic} under different assignment rules, focusing on the asymptotic values of $c_t$. We remark that in this model, the WGC cycle cannot appear: the new vehicle appears immediately once the previous one gets assigned, keeping the size of the available fleet constant. By these means, for now we focus solely on the IGD.

The sequence $(S^t)_{t \geq 1}$ describes the evolution of the positions of the $N$ drivers and is an homogeneous Markov chain with state space $X^N$.  This state space is finite if $X$ is finite, countable if $X$ is countably infinite, and uncountable if $X$ is uncountable. Throughout the paper we consider both discrete (finite) and continuous (uncountable) cases and clarify when needed. For now, because Markov chains are much simpler when the state space is finite or countable, we include this extra assumption for $X$. This is largely harmless from an interpretation perspective, as any transport scenario will ultimately involve a finite  (potentially discretised) set of feasible locations. Nevertheless, in section \ref{sec:1d} we analyse the one-dimensional circle considering both discrete and continuous cases, whereas in section \ref{sec:Square} we present visualisations on the continuous two-dimensional square. Also, in the Appendix we provide some theoretical support for models with a continuous and compact ground set $X\subseteq\mathbb{R}^d$.

In this discrete scenario for $X$, we assume that users and drivers may appear anywhere so that for every \ location $u \in X$ we have \ $Q(x_t=u)>0$ and $P(y_t=u)>0$. 
This implies that the Markov chain is:
\begin{enumerate}
    \item \textit{Irreducible}: This property means that if $A,B\in X^N$ are two $N$-tuples in the state space $X^N$, there exists some integer $k$ such that $\mathbb{P}(S^{t+k}=B|S^t=A)>0$. In other words, regardless of the current state of the system, any other state can be reached in finite time. In our setting, this can always be achieved in $N$ or less steps. Indeed, it suffices that first, in iteration $t$, the new user $x_t$ appears exactly in $A_1$, and the new driver $y_t$ appears in $B_1$; this would imply that the first element of $A$ would be swapped by the first element of $B$. In the subsequent $N-1$ iterations, the same would happen but with the remaining coordinates one by one, until reaching the vector $B$ at iteration $t+N$.
    \item \textit{Aperiodic}: This follows from the fact that for all $A \in X^N$ we have $\mathbb{P}(S^{t+1}=A|S^t=A)>0$. The latter is true by a similar argument as above: it suffices that both the user $x_t$ and driver $y_t$ appear where the first driver used to be, i.e. $x_t=y_t=A_1$.
\end{enumerate}

As a direct application of the  convergence theorem for finite ergodic  Markov Chains, these two properties imply the following lemma.
\begin{lemma}\label{Lemma:Ergodic}
    If $X$ is finite, the chain has a unique  invariant probability measure $\pi$ such that $\sum_{A \in X^N} \pi_A =1$ and  regardless of the starting point $S^1=s$, we have 
$$\lim_{t\rightarrow\infty} \mathbb{P}(S^t=A|S^1=s)=\pi_A.$$
\end{lemma}
\noindent Note that the final equilibrium is not a specific location tuple of the drivers, but probabilities over all possible location tuples. We refer to this ergodic equilibrium as the \textit{IGD Equilibrium}. Using Lemma \ref{Lemma:Ergodic}, standard results from Markov Decission Processes imply the following theorem \citep{puterman2014markov}:

\begin{theorem}\label{thm:CostsConverge}
$\mathbb{E}(c_t)$ converges as $t \rightarrow \infty$.     
\end{theorem}
Under mild regularity assumptions, these results can be extended to the case  where $X$ is a  compact subset of $\mathbb{R}^d$ and the probabilities governing users and drivers have a  strictly positive density
with respect to the Lebesgue measure. The proofs in this case are  more technical and follow  by combining results from \cite{meyn2012markov} and \cite{lasserre2000invariant},
which ensure the convergence in \textit{total variation} of the  distribution of the chain towards a unique invariant measure. The details are presented in the Appendix.

\section{The One-dimensional Circle}\label{sec:1d}

In order to get a better intuition of the dynamics, as well as various theoretical results, we now focus on a simple case, which is still challenging to analyse. We consider the interval $[0,1]$, and $d(u,v)=\min(|v-u|,1-|v-u|)$, representing that the points $0$ and $1$ are merged, or that the process takes place in a circle of length 1 (just the border). We assume both drivers and users appear following a uniform distribution and independently. Throughout this section, we will detail when we are considering the full interval (continuous), or a discretisation where $X=\{M/N: M=1\ldots,N\}$.

Note that we can assume that drivers are sorted, simply by reassigning the subindices. That is, we assume $s_1^t \leq s_2^t \leq \ldots \leq s_N^t$, and thus define $\ell_k^t=d(s_k^t,s_{k+1}^t)$, with the convention that $s^t_{N+1}=s^t_1$. In plain words, $\ell_k^t$ is the length of the interval $I_k^t$ between the two consecutive drivers $s^t_k,s^t_{k+1}$. Note that $\ell$ is a function of $s$, but we omit that in the notation to keep it lighter. This is all illustrated in Fig. \ref{fig:example_ring}. 
These intervals constitute the fundamental object in this single-dimensional model, and the Markov chain can be described in terms of $\ell^t=(\ell^t_1,\ldots,\ell^t_N)$ as the state.

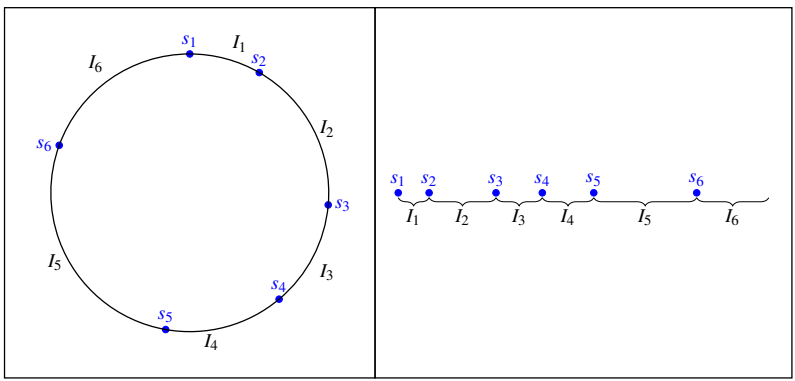
\begin{figure}[ht]
    \centering
    \resizebox{0.6\textwidth}{!}{  
 \usetikzlibrary{decorations.pathreplacing}

\begin{tikzpicture}

\def\angleSone{90}   
\def\angleStwo{60}   
\def\angleSthree{-5} 
\def\angleSfour{-50} 
\def\angleSfive{-100} 
\def\angleSsix{160}  
\def\angleSsixp{-200} 

\draw[thick] (-8,-4) rectangle (0,4); 

\filldraw[blue] ({-4+cos(\angleSone)*3},{sin(\angleSone)*3}) circle (2pt) node[above] {\(s_1\)};
\filldraw[blue] ({-4+cos(\angleStwo)*3},{sin(\angleStwo)*3}) circle (2pt) node[above] {\(s_2\)};
\filldraw[blue] ({-4+cos(\angleSthree)*3},{sin(\angleSthree)*3}) circle (2pt) node[right] {\(s_3\)};
\filldraw[blue] ({-4+cos(\angleSfour)*3},{sin(\angleSfour)*3}) circle (2pt) node[above] {\(s_4\)};
\filldraw[blue] ({-4+cos(\angleSfive)*3},{sin(\angleSfive)*3}) circle (2pt) node[above] {\(s_5\)};
\filldraw[blue] ({-4+cos(\angleSsix)*3},{sin(\angleSsix)*3}) circle (2pt) node[left] {\(s_6\)};

\draw[thick] ({-4+cos(\angleSone)*3},{sin(\angleSone)*3}) arc[start angle=\angleSone, end angle=\angleStwo, radius=3] node[midway, above right] {\(I_1\)};
\draw[thick] ({-4+cos(\angleStwo)*3},{sin(\angleStwo)*3}) arc[start angle=\angleStwo, end angle=\angleSthree, radius=3] node[midway, right] {\(I_2\)};
\draw[thick] ({-4+cos(\angleSthree)*3},{sin(\angleSthree)*3}) arc[start angle=\angleSthree, end angle=\angleSfour, radius=3] node[midway, below right] {\(I_3\)};
\draw[thick] ({-4+cos(\angleSfour)*3},{sin(\angleSfour)*3}) arc[start angle=\angleSfour, end angle=\angleSfive, radius=3] node[midway, below left] {\(I_4\)};
\draw[thick] ({-4+cos(\angleSfive)*3},{sin(\angleSfive)*3}) arc[start angle=\angleSfive, end angle=\angleSsixp, radius=3] node[midway, left] {\(I_5\)};
\draw[thick] ({-4+cos(\angleSone)*3},{sin(\angleSone)*3}) arc[start angle=\angleSone, end angle=\angleSsix, radius=3] node[midway, above left] {\(I_6\)};

\draw[thick] (0,-4) rectangle (9,4); 

\def\totalLength{8}  
\def\distanceOne{(\angleSone - \angleStwo)/360*\totalLength}  
\def\distanceTwo{(\angleStwo - \angleSthree)/360*\totalLength}  
\def\distanceThree{(\angleSthree - \angleSfour)/360*\totalLength}  
\def\distanceFour{(\angleSfour - \angleSfive)/360*\totalLength}  
\def\distanceFive{(\angleSfive - \angleSsixp)/360*\totalLength} 
\def\distanceSix{(360 - (\angleSsix + \angleSone))/360*\totalLength}  

\filldraw[blue] (0.5,0) circle (2pt) node[above] {\(s_1\)};
\filldraw[blue] ({0.5 + \distanceOne},0) circle (2pt) node[above] {\(s_2\)};
\filldraw[blue] ({0.5 + \distanceOne + \distanceTwo},0) circle (2pt) node[above] {\(s_3\)};
\filldraw[blue] ({0.5 + \distanceOne + \distanceTwo + \distanceThree},0) circle (2pt) node[above] {\(s_4\)};
\filldraw[blue] ({0.5 + \distanceOne + \distanceTwo + \distanceThree + \distanceFour},0) circle (2pt) node[above] {\(s_5\)};
\filldraw[blue] ({0.5 + \distanceOne + \distanceTwo + \distanceThree + \distanceFour + \distanceFive},0) circle (2pt) node[above] {\(s_6\)};

\draw [decorate,decoration={brace,amplitude=5pt,mirror,raise=3pt},yshift=0pt] 
(0.5,0) -- ({0.5 + \distanceOne},0) node[midway, below=6pt] {\(I_1\)};
\draw [decorate,decoration={brace,amplitude=5pt,mirror,raise=3pt},yshift=0pt] 
({0.5 + \distanceOne},0) -- ({0.5 + \distanceOne + \distanceTwo},0) node[midway, below=6pt] {\(I_2\)};
\draw [decorate,decoration={brace,amplitude=5pt,mirror,raise=3pt},yshift=0pt] 
({0.5 + \distanceOne + \distanceTwo},0) -- ({0.5 + \distanceOne + \distanceTwo + \distanceThree},0) node[midway, below=6pt] {\(I_3\)};
\draw [decorate,decoration={brace,amplitude=5pt,mirror,raise=3pt},yshift=0pt] 
({0.5 + \distanceOne + \distanceTwo + \distanceThree},0) -- ({0.5 + \distanceOne + \distanceTwo + \distanceThree + \distanceFour},0) node[midway, below=6pt] {\(I_4\)};
\draw [decorate,decoration={brace,amplitude=5pt,mirror,raise=3pt},yshift=0pt] 
({0.5 + \distanceOne + \distanceTwo + \distanceThree + \distanceFour},0) -- ({0.5 + \distanceOne + \distanceTwo + \distanceThree + \distanceFour + \distanceFive},0) node[midway, below=6pt] {\(I_5\)};
\draw [decorate,decoration={brace,amplitude=5pt,mirror,raise=3pt},yshift=0pt] 
({0.5 + \distanceOne + \distanceTwo + \distanceThree + \distanceFour + \distanceFive},0) -- (8.5,0) node[midway, below=6pt] {\(I_6\)}; 

\end{tikzpicture} 
}    \caption{Two interpretations of the metric space, either as a circle (left), or as the $[0,1]$ interval where 0 and 1 are identified as the same.}
    \label{fig:example_ring}
\end{figure}

Let us focus for now on the continuous setting. In this case, we can easily compute the expected value of the costs given the positions of the drivers. First,  the probability that the new passenger appears in the $k$-th interval $I_k^t$ is exactly $\ell_k^t$. Second, it is easy to see that its expected cost conditional on appearing in the $k$-th interval is $\ell_k^t/4$. Therefore, defining the function $V(\ell)=\sum_{k=1}^N \ell_k^2$ for any vector $\ell=(\ell_1,\ldots,\ell_N)$, and conditioning on the possible intervals where the new user can appear, it follows that the expected cost at stage $t$ is 
\begin{equation}\label{Eq:CostEqualVar}
    \mathbb{E}\big[c_t|\ell^t\big] = \sum_{k=1}^N \ell_k \cdot \frac{\ell_k}{4}= \frac{V(\ell^t)}{4}.
\end{equation}

Note that $V(\ell)$ is minimised when $\ell_k=\frac{1}{N}$ for all $k$. That is, the better distributed the drivers, the lower the expected cost. This is consistent with the intuition that an equidistant distribution of the drivers is the optimal situation.

In order to understand the dynamics of the system, let us analyse what is expected to happen in the next iteration given the current one. Concretely, given $\ell^t=\ell$ we analyse the expected cost in $t+1$. When a driver is assigned and thus removed, a new \textit{merged} interval appears, as illustrated in Fig. \ref{fig:MERGED}. Two cases can happen, as the new driver might appear in this same merged interval or in a different one. A careful computation of both cases allows us to compute the function $\Delta V(\ell)$, defined as the expected change in $V(\cdot)$
from $\ell$ to the next random state $\ell'$ that arises after the arrival of a user and the replacement of the assigned vehicle by a new vehicle at a random position.

\begin{figure}[ht]
    \centering   
\includegraphics[width=60mm]{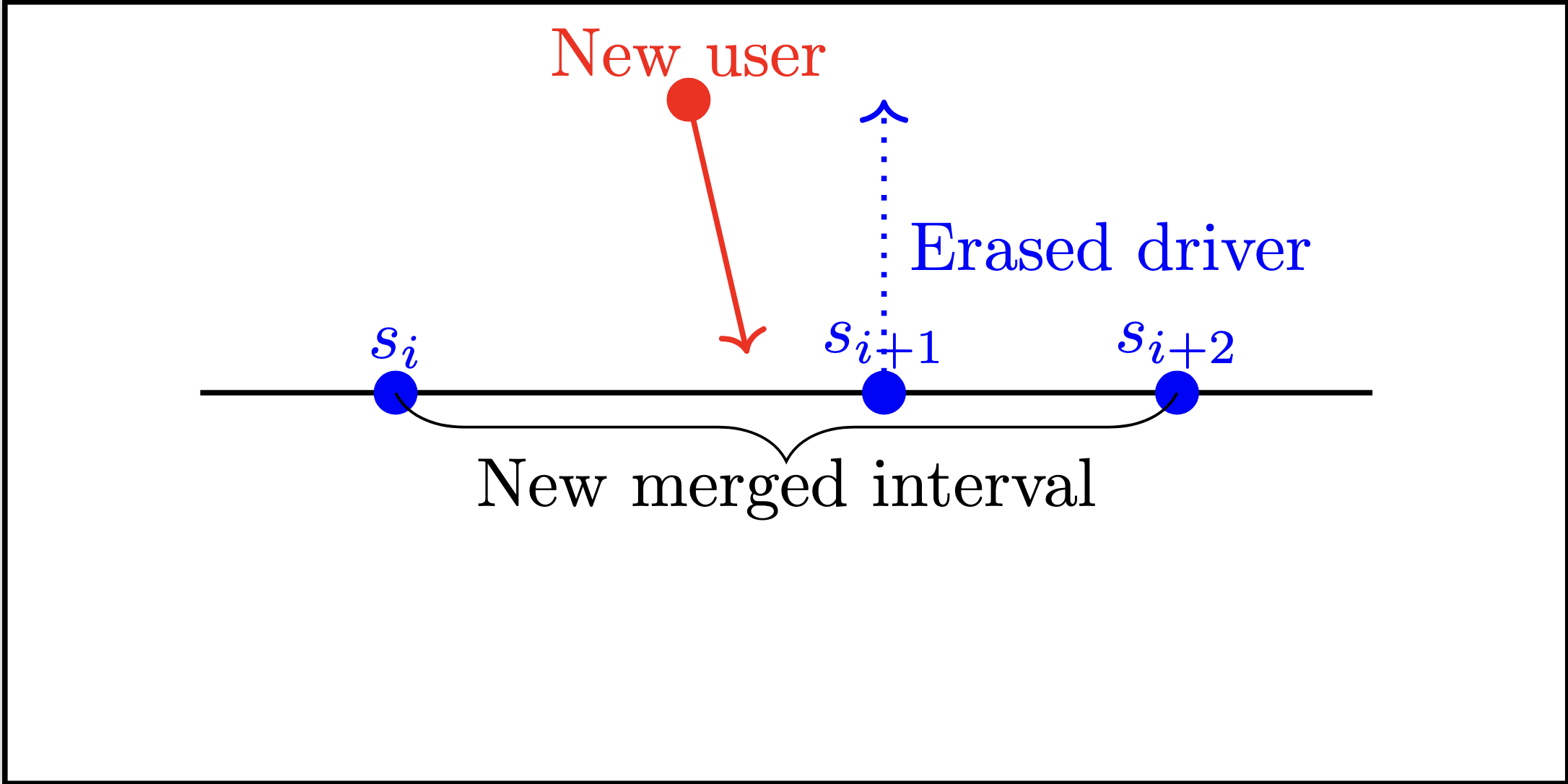}
    \caption{When a new user arrives, the two intervals that surround the assigned vehicle $s_{i+1}$ get merged.}
    \label{fig:MERGED}
\end{figure}

\begin{lemma}\label{Lemma:ChangeInIV}
In the continuous case:
\begin{multline}\label{ChangeInV}
    \Delta V(\ell) := \mathbb{E}\big[V(\ell')-V(\ell)\,|\,\ell\big] = \\
    \underbrace{\sum_{k=1}^N \frac{1}{3} \left[4\ell_k\ell_{k+1} - \ell_k^2 - \ell_{k+1}^2 \right]\frac{(\ell_k + \ell_{k+1})}{2} (\ell_k+{\ell_{k+1}})}_{(^*1)} + 
    \underbrace{\sum_{k=1}^N \sum_{j \neq \{k,k+1\}} 2\left[\ell_k \ell_{k+1} - \frac{\ell_j^2}{6}\right]\frac{(\ell_k + \ell_{k+1})}{2}\ell_j}_{(^*2)}
\end{multline}
\end{lemma}

In Eq. \eqref{ChangeInV}, the term (*1) represents the case where the new driver appears in the merged interval, and (*2) when it appears in a different interval. The expectation is taken over all possible appearances of the new user and vehicle.
The proof of Lemma \ref{Lemma:ChangeInIV} is provided in the appendix.

 Eq. \eqref{ChangeInV} can be used to prove one of the fundamental insights of this paper, namely that \textbf{a fully random distribution tends to worsen under the IGD}. To be precise:
 
 \begin{theorem} \label{thm:IGDWorseThanUniform}
In the continuous case, if the vector $S$ follows a uniform distribution (each of its coordinates is uniformly and independently drawn from $[0,1]$), then
\begin{equation}
    \mathbb{E}_\ell\big[\Delta V(\ell(s))\big]>0
\end{equation}
Where $\mathbb{E}_\ell[\cdot]$ represents that the expectation is taken over $\ell$.

 \end{theorem}

The proof is provided in the appendix. Its core is to show that:
\begin{itemize}
    \item (*1) is zero. This has a direct interpretation. The merged interval is everything that matters for the change, as all the other terms in $V(\ell)$ will remain the same. The merged interval used to be divided by a random driver and now is again randomly divided by a new driver, which is why the expected change is zero.
    \item (*2) is positive, which can also be interpreted. Ideally, the new user would appear in a short interval and the new driver in a long one. That would result in an elongation of the short interval and a subdivision of the long one, improving the overall situation. The opposite case, where a short interval is subdivided and a long interval is further increased, would be worse.
    
    To be neutral, let us focus on the case where both intervals were of the same size. In that case, the interval of the new user is merged with a different one, increasing its length, whereas the interval of the new driver is divided into two smaller ones. In total, the distribution of the interval becomes more heterogeneous, i.e., worse. This is illustrated in Fig. \ref{fig:OneIntervalMergedTheOtherDivided}.

\begin{figure}[ht]
\centering
  \makebox[\textwidth]{ 
    \resizebox{0.8\textwidth}{!}{
          \usetikzlibrary{decorations.pathreplacing, arrows.meta}

\begin{tikzpicture}[scale=1.1]

\draw[thick] (0,-2) rectangle (8,2);
\node at (4,2.3) {\textbf{Before}};

\draw[thick] (1,0) -- (3,0); 
\draw[thick] (5,0) -- (7,0); 

\filldraw[blue] (1,0) circle (3pt) node[below] {\(s_i\)};
\filldraw[blue] (3,0) circle (3pt) node[below] {\(s_{i+1}\)};
\filldraw[blue] (5,0) circle (3pt) node[below] {\(s_j\)};
\filldraw[blue] (7,0) circle (3pt) node[below] {\(s_{j+1}\)};

\filldraw[red] (2,1.2) circle (3pt) node[above] {New user};
\draw[red, ->, thick] (2,1.2) -- (2,0.2);

\definecolor{brown}{RGB}{139,69,19}
\filldraw[brown] (6,-1.2) circle (3pt) node[below] {New driver};
\draw[brown, ->, thick] (6,-1.2) -- (6,-0.2);

\draw[thick] (9,-2) rectangle (17,2);
\node at (13,2.3) {\textbf{After}};

\draw[thick] (10,0) -- (13,0); 
\filldraw[blue] (10,0) circle (3pt) node[below] {\(s_i\)};
\filldraw[blue] (13,0) circle (3pt) node[below] {\(s_{i+2}\)};
\draw [decorate,decoration={brace,amplitude=10pt},yshift=6pt] 
(10,0) -- (13,0) node[midway, yshift=17pt] {Merged};

\draw[thick] (14,0) -- (16,0);
\filldraw[blue] (14,0) circle (3pt) node[below] {\(s_j\)};
\filldraw[blue] (15,0) circle (3pt) node[below] {\(s_{new}\)};
\filldraw[blue] (16,0) circle (3pt) node[below] {\(s_{j+1}\)};
\draw [decorate,decoration={brace,amplitude=10pt},yshift=6pt] 
(14,0) -- (16,0) node[midway, yshift=17pt] {Subdivided};

\end{tikzpicture} 
    }
  }
  \caption{In the left, two intervals had the same length. A new user and a new driver appear, each on one of these intervals. As a result, the interval where the user appeared merges with another one and becomes wider, while the interval where the driver appeared becomes subdivided. The ``After'' situation is therefore worse than the ``Before''.}
  \label{fig:OneIntervalMergedTheOtherDivided}
\end{figure}
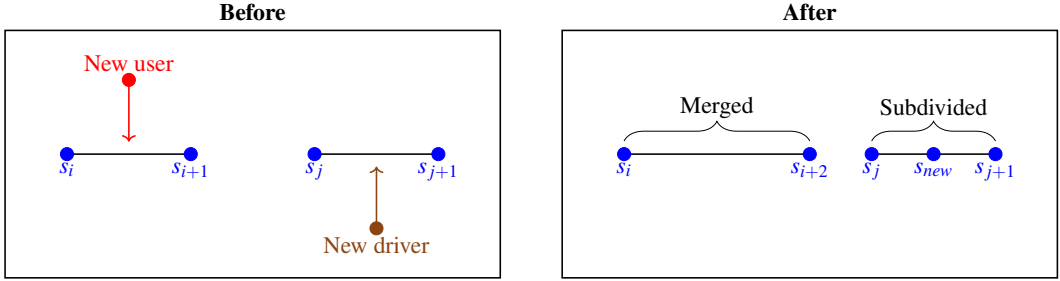
    
\end{itemize}

A crucial point is that the reasoning behind steps (*1) and (*2) follows the exact same logic illustrated in Fig. \ref{fig:ExampleIntro}. Theorem \ref{thm:IGDWorseThanUniform} suggests that the IGD limit distribution should perform worse than a fully random distribution of drivers, an observation that is empirically confirmed below. This outcome stems directly from the IGD dynamics: when a user appears, it often leads to the creation of a long interval without drivers, as discussed when analysing (*2). Such long intervals have a significant effect on $V(\ell)$ because (i) they have a high probability of containing a future user, and (ii) users appearing within those intervals incur a high expected cost.

{

The result from Theorem \ref{thm:IGDWorseThanUniform} also holds empirically in the discrete case. However, providing a formal proof is significantly more challenging, as the closed-form expression used for the probabilistic distribution of the $\ell_i$ is only valid in the continuous setting. In the discrete case, Eq. \eqref{Eq:CostEqualVar} requires a slight correction -- specifically, subtracting $O/4M^2$, where $O$ is the number of intervals with odd length. This correction is always small, as it is bounded above by $1/4M$.

In Fig. \ref{fig:ComparisonIGPVSUniformVSRandom}, we compare the IGD against two benchmarks, namely the static expected cost if drivers are either random (red curves) or uniformly distributed (green curves). In the left, we show the limit of $\mathbb{E}\big[c_t\big]$ for different numbers of drivers $N$, calculated through Monte Carlo simulations. Therein, it becomes clear that IGD yields significantly worse results than the two benchmarks.  While the random and uniform curves decrease proportional to $1/N$, our simulations suggest that the IGD equilibrium decreases at a rate that lies between $\log(N)/N$ and $\log^2(N)/N$.

The dynamic component can be seen on in the right panel of Fig. \ref{fig:ComparisonIGPVSUniformVSRandom}, where we fix $N=50$. For IGD, we compute (through Monte Carlo simulations) the evolution of the expected cost faced by each user if drivers start uniformly distributed. The increasing trend until convergence is also evident. Theorem \ref{thm:IGDWorseThanUniform} can be observed by noting that the limit of the blue curve is significantly greater than what would be achieved if the drivers remained randomly distributed (red horizontal line).

\begin{figure}[ht]
    \centering
    \begin{subfigure}{0.35\textwidth}
        \centering
        \includegraphics[width=\linewidth]{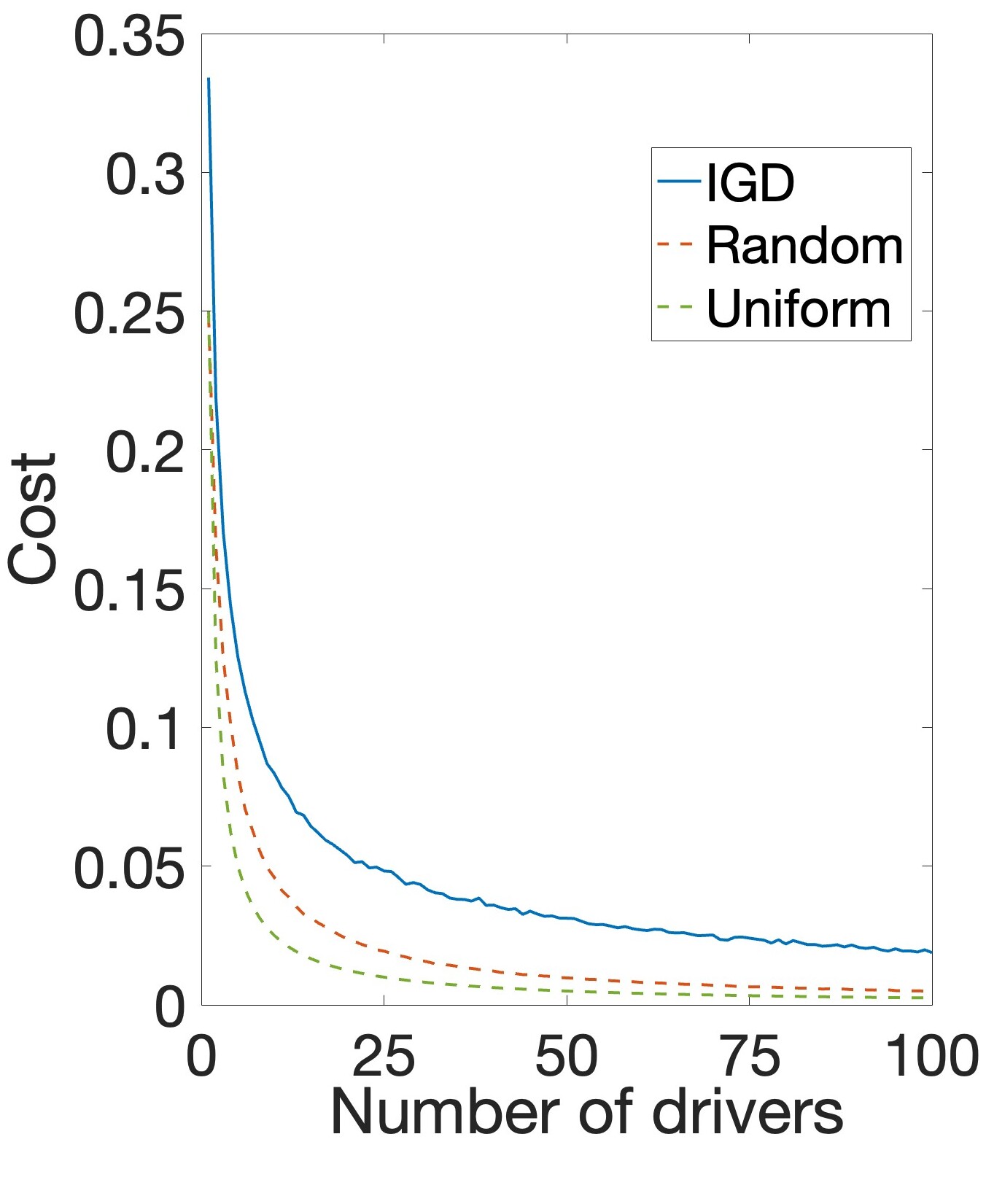}
    \end{subfigure}
    \begin{subfigure}{0.37\textwidth}
        \centering
        \includegraphics[width=\linewidth]{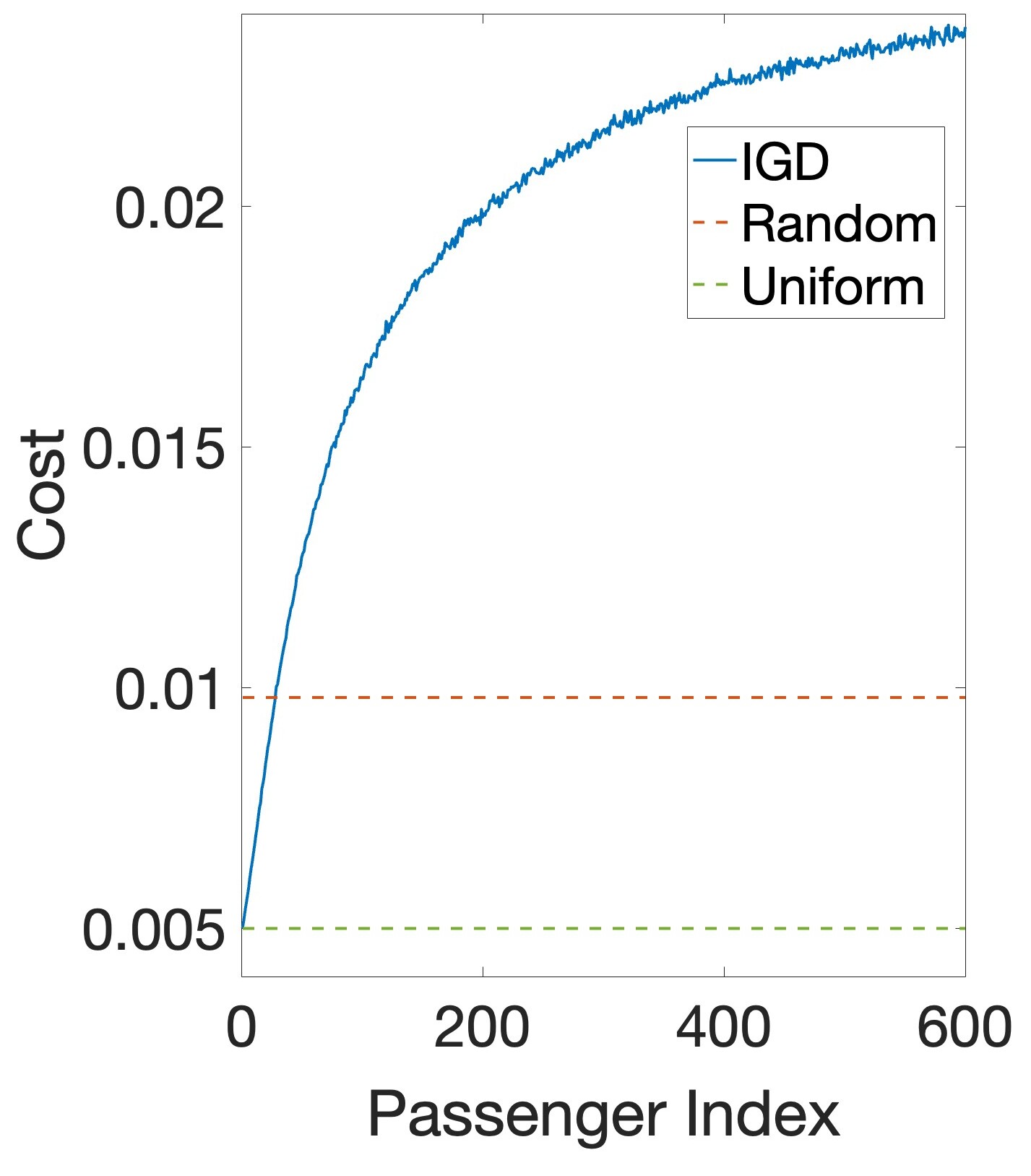}
    \end{subfigure}
    \caption{Comparison of the IGD with having the drivers randomly or uniformly located.}
    \label{fig:ComparisonIGPVSUniformVSRandom}
\end{figure}

We can further illustrate the source of why the IGD equilibrium is different than a random distribution through the following hypothetical example. Imagine that instead of assigning the users greedily, we do it randomly. In this case, the users' position would play no role, or in other words, we replace one random driver with another random driver. Therefore, that system would converge to a fully random distribution. The difference between assigning greedily and randomly lies precisely on the gaps: the drivers that already cover a large area have a greater probability of being selected, thus increasing their area even further. This is the ultimate reason why the IDG equilibrium is worse than random. 

A clarification is relevant in the example above. Assigning randomly would obviously yield greater costs, because every assignment would be very costly. A random assignment would be better than the IGD equilibrium for the next user, if that user is assigned greedily, but not for the current one.

Similarly, there is a noteworthy analogy with the well-known \textit{inspection paradox} \citep{stein1985sampling}.  This refers to a statistical bias that arises when the sampling method is related to the unit being sampled. The most famous example in transport is the difference between time-mean-speed TMS and space-mean-speed SMS \citep{knoop2009empirical}. In TMS, fast vehicles are sampled more frequently due to their higher speed, resulting in an overestimation of the true mean speed SMS. In the IGD, we are not sampling or estimating; instead, we are choosing a driver in every iteration. The probability of selecting a given driver is greater if the driver covers a larger dominance zone, which in turn increases the resulting cost. Similar to the difference between TMS and SMS, the resulting cost is greater than if every driver would have the same probability of being chosen.

\subsection{Some theoretical results on optimal assignments, and a comparison with the greedy rule} \label{sec:Greedy?}

So far, we have assumed a greedy rule when assigning users to drivers. In this subsection, we argue why this is a reasonable matching mechanism, and more specifically, why a better one would not prevent the IDG from happening\footnote{The literature reports related results in similar settings. Let us denote by $c(N)$ the limit cost for $\mathbb{E}\big[c_t\big]$ when we have $N$ drivers. In the continuous case, \cite{kanoria2021dynamic} proves that there exists $K$ such that no policy can achieve $c(N) \leq K \log(N)/N$. Note that if drivers would remain uniformly distributed, then $c(N) = 0.25/N$. Numerically, we find that greedy performs better than $\log^2(N)/N$ (more specifically, that $\frac{c(N)}{\log^2(N)/N} \rightarrow 0$ when $N\rightarrow \infty$), i.e., greedy is not worse than $\log(N)$ times the uniform result, which is obviously better than what can be achieved with an optimal policy. This resonates with the theoretical results for the greedy algorithm by \cite{balkanski2023power} in a similar setting, but when all drivers are available from the beginning and disappear one by one, and where they prove that the ratio between greedy and the optimal solution is not worse than a constant.}. As we are discussing the (sub)optimality of matching mechanisms, it is worth noting that our system can be described as a Markov Decision Process (MDP), where we are searching for a \textit{policy}, i.e., a function $\rho:[0,1]^{N+1} \rightarrow \{1,\ldots,N\}$ that takes as inputs the current position $(s_1,\ldots,s_N,u)$ of the drivers and the user, and returns which driver is assigned to the user. The greedy rule is one possible such $\rho$. The question is which is the optimal $\rho$, i.e., the one minimising the limit of the average cost when $t\rightarrow \infty$. In the \textit{continuous} case, where drivers and users can appear in any position, an optimal policy might not exist. In the \textit{discrete} case, where the circle is subdivided into $M$ subintervals of length $1/M$ each, there is a finite number of policies, so one of them must be optimal. We now describe our results.

\subsubsection{General results about the optimal policies}
\begin{theorem}
    If there are two drivers, greedy is optimal. This is valid in the discrete and in the continuous case.
\end{theorem}

\begin{proof}
    We note that when a user appears, the system's status after the assigned driver disappears is equivalent regardless of which driver was assigned. Namely, it has one driver located somewhere in the circle, and another driver will appear randomly. This implies that the future costs of the system are independent of the current assignment, and therefore a greedy assignment is optimal.
\end{proof}

The previous argument does not hold when the number of drivers exceeds two. Even with just three drivers, if two are close to each other but the third is not, under certain conditions it may be preferable to incur a slightly higher cost now by assigning one of the two drivers that are clustered together. However, there is a remarkable result that is valid for any $N$, namely, that an optimal assignment rule should always match to one of the two \textit{neighbouring} drivers.

In order to formalise this, let us consider a user $u$ and a vector of drivers $s$. We define $R(u,s)$ as the first driver to the right of $u$, and $L(u,s)$ as the first driver to the left. If there is a driver exactly at the same position as $u$ (which might happen in the discrete case), we have $R(u,s)=L(u,s)=u$. We call $R(u,s)$ and $L(u,s)$ the ``neighbours'' of $u$.

\begin{theorem}\label{Thm:NeihbourDrivers}
Consider the discrete case, where the circle is partitioned into $M$ sub-intervals of length $1/M$ each. 
Then any optimal policy will always assign a user $u$ to either $R(u,s)$ or $L(u,s)$.
As a consequence, at least half of the actions of any optimal policy are greedy.
These properties hold regardless of whether the arrival probabilities of users and drivers are uniform or not.
\end{theorem}

\begin{proof}
Consider a policy $\rho$ that does not always assign to a neighbour. We will show that $\rho$ can be improved, implying that it is not optimal. We will construct a \textit{non-Markovian} policy, meaning one that depends not only on the current state but also on past states. This does not pose a problem, since in any finite MDP, the expected value achieved by any non-Markovian policy can also be attained by a Markovian policy \citep{puterman2014markov}. In other words, it suffices to show that there exists a non-Markovian policy that is better than $\rho$.

This non-Markovian policy $\rho'$ operates as follows. It has a  passive/active label that begins as passive. It assigns exactly the same as $\rho$ for the initial iterations (potentially none) until the first time $\rho$ does not assign to a neighbour. Let us denote by $s_a$ the driver that was assigned by $\rho$, and by $s_n$ the neighbour that is \textit{between} $s_a$ and $u$. That is, $s_a$ is to the left of $u$ if the distance between $s_a$ and $u$ is achieved by $u$ moving to the left, and vice-versa; if the distance is exactly $1/2$, we can take $s_n$ to be any of the two neighbours. In this iteration, $\rho'$ assigns to $s_n$ instead of $s_a$. This is shown in Fig. \ref{fig:NonMarkovianPolicyNeighbours}. We now say that $\rho'$ is active.

\begin{figure}[ht]
\centering
  \makebox[\textwidth]{ 
    \resizebox{0.55\textwidth}{!}{ 
      \begin{tikzpicture}[scale=1.1]

\draw[thick] (0,0) -- (8,0);

\filldraw[blue] (2,0) circle (3pt) node[below] {\(s_a\)};

\filldraw[blue] (5,0) circle (3pt) node[below] {\(s_n\)};

\filldraw[red] (7,0) circle (3pt) node[below] {\(u\)};

\draw [decorate,decoration={brace,amplitude=10pt},yshift=6pt]
(2,0) -- (7,0) node[midway, yshift=14pt] {\(\leq 1/2\)};

\end{tikzpicture} 
    }
  }
  \caption{The policy $\rho$ assigns $u$ to $s_a$, so $\rho'$ assignes its neighbour $s_n$ instead.}
  \label{fig:NonMarkovianPolicyNeighbours}
\end{figure}
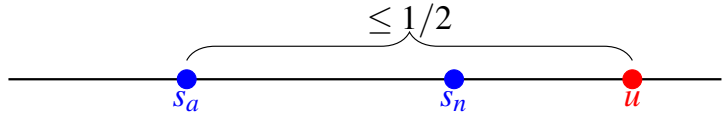

Note that after this assignment, $\rho$ and $\rho'$ do not have the same set of drivers, but they differ by one: $s_n$ would be available if using $\rho$, but we have $s_a$ available instead. In the next iterations (potentially none), $\rho'$ copies what $\rho$ would do if having $s_n$. This is stopped when $\rho$ would assign to $s_n$, in which case $\rho'$ assigns to $s_a$. That is:

\begin{equation}
    \rho'(s,u)=\begin{cases}
       \rho(s\cup \{s_n\} \setminus \{s_a\},u) & \text{ if }  \rho(s\cup \{s_n\} \setminus \{s_a\},u) \neq s_n \\
       s_a & \text{ Otherwise}      
    \end{cases}
\end{equation}

When the latter happens, $\rho'$ becomes passive and the \textit{cycle} starts again. Note that in a whole cycle, i.e., since $\rho'$ becomes passive (or the first iteration) until it becomes passive again, $\rho$ and $\rho'$ have the same cost in every iteration but two. Denoting by $u_1$ and $u_2$ the users involved in these two iterations, we have that the cost using $\rho$ would be $d(u_1,s_a) + d(u_2,s_n)$. And then

\begin{equation}\label{Eq:DifferencePP'}
    d(u_1,s_a) + d(u_2,s_n) = d(u_1,s_n) + d(s_n,s_a) + d(u_2,s_n) \geq d(u_1,s_n) + d(u_2,s_a)
\end{equation}
where the first equality in Eq. \eqref{Eq:DifferencePP'} holds because $s_n \in (u_1,s_n)$, and the inequality is just the triangular inequality. Note that the right hand side $d(u_1,s_n) + d(u_2,s_a)$ is exactly the cost yielded by $\rho'$. In other words, $\rho'$ yields a cost that is lower or equal than $\rho$ in every cycle. Moreover, there is a positive probability that the inequality is strict: for instance, if $u_2=s_a$ the inequality would be strict, and this has a positive probability. It suffices that all new users and drivers appear in the same position as $s_a$ from the moment $\rho'$ becomes active until it becomes passive. This completes the proof that $\rho'$ is better than $\rho$, and as argued above, we conclude that any optimal policy must always assign to a neighbour.

We now show why at least half of the actions are greedy. To do this, note that if $d_1$ and $d_2$ are two consecutive drivers, then we can divide the interval $[d_1,d_2]$ into two subintervals $[d_1,z],(z,d_2]$ for some $z$, such that if a user appears in $[d_1,z]$ it would be assigned to $d_1$, and a user in $(z,d_2]$ would be assigned to $d_2$ (to ease the notation, we are using $[a,b]$ to describe all the discrete slots that belong in that interval). To see why such a $z$ exists, suppose instead there exist users $u_1 < u_2$ with $u_1$ assigned to $d_2$ and $u_2$ assigned to $d_1$. Then, swapping these assignments would keep the probabilistic distribution for the next iteration unchanged, and diminish the expected costs in this iteration, which cannot happen in an optimal policy. This division into two subintervals directly implies that at least half of the users' positions would be assigned greedily.
\end{proof}

Let us remark that a similar result holds in the continuous case. While there might be no optimal policy, any policy that sometimes does not assign to a neighbour can be improved doing exactly the same as we did in the proof of Theorem \ref{Thm:NeihbourDrivers}, i.e., building $\rho'$. However, this change would require admitting Non-Markovian policies, as the result mentioned above from \cite{puterman2014markov} is not necessarily valid when the state space is uncountable.

A natural question is whether a similar Theorem \ref{Thm:NeihbourDrivers} holds in other domains? The very notion of ``neighbours'' only makes sense in one dimension. On the other hand, during the proof, the fact that we were considering neighbours was only used in Eq. \eqref{Eq:DifferencePP'}, specifically in the first equality, as we knew that \textit{the shortest path from the assigned driver $s_a$ to the user $u_1$ passed through the neighbour $s_n$.} This means that in other spatial domains, whenever the same situation happens, we can discard $s_a$ as an assignable driver. In Euclidean regions in two or more dimensions, this would typically have probability zero, as it requires all the points to lie on a one-dimensional line. But in graphs, this can help. All together, this implies the following corollary:

\begin{corollary}\label{Cor1}
Consider $X=(V,E)$ a graph (directed or undirected) where every arc has an associated distance. Take an optimal policy $\rho$. If the user is in $u \in V$, and there are vehicles in $v_1, v_2 \in V$, with $v_1$ lying within a shortest path from $v_2$ to $u$, then $\rho$ will not assign $v_2$ to $u$.
\end{corollary}

While Corollary \ref{Cor1} seems too specific, there are some types of graphs where it can be useful to reduce the number of potential actions. Two notable examples are:

\begin{enumerate}
    \item Take $G$ an undirected grid, and $u$ a node there. If $v_1$ and $v_2$ are both to the north and east of $u$, but $v_2$ is further in both dimensions, then $v_2$ can be discarded. Obviously, the same happens if replacing ``north'' with ``south'', and/or ``east'' with ``west''.
    \item Take $G$ a tree, and define its root to be where the user is located. If $v_1$ is an ancestor of $v_2$, then $v_2$ can be discarded when deciding which vehicle to assign.
\end{enumerate}

Theorem \ref{Thm:NeihbourDrivers} has several implications. The first one is about reinforcing the intuition that greedy is reasonable. Indeed, while we should not always assign greedily, we should not do something too different either. We need to always select one of the two neighbours, and greedy selects the closest one. 

There is also a direct relationship between Theorem \ref{Thm:NeihbourDrivers} and the IGD. \textbf{If we use the optimal policy, and a user arrives in a long gap, this long gap will increase} because we will remove one of the gap's borders. This can be seen using Fig. \ref{fig:MERGED} again: the optimal policy would assign either $s_i$ or $s_{i+1}$ to the new user, and in any case the long interval $[s_i,s_{i+1}]$ would get merged with another interval.

On the other hand, Theorem~\ref{Thm:NeihbourDrivers} significantly narrows the policy space by showing that, in each state, at most two assignments need to be considered instead of all $N$ possibilities. This reduction makes it feasible to explore optimal policies directly. In the next subsection, we leverage this result to compute the optimal policy for small values of $M$ and $N$, compare its performance with the greedy policy, and confirm that the IGD phenomenon still emerges, even under optimal decision-making.

\subsubsection{Explicit optimal policies for small instances of the discrete case}\label{sec:OptimalSmallInstances}

In any finite MDP, the optimal policy can be derived through a Linear Program \citep{puterman2014markov}. This LP has as many variables as the number of possible states, and as many restrictions as the number of states multiplied by the number of actions for each state, which in our case is lower or equal than two (thanks to Theorem \ref{Thm:NeihbourDrivers}). In practice, we can solve this problem for some combinations of $M \leq 15, N \leq 7$. 

Let us first fix $M=8$, and summarise the results in Table \ref{tab:GreedyVSOptimalM8}. For $N=3,\ldots,7$,  we show two results:
\begin{itemize}
    \item In the top row, we have the extra cost yielded by greedy in percentages (computed through Monte Carlo simulations), compared to the optimal policy. While these values increase, they do it concavely, and moreover, the percentage is never over 10\%. This suggests that greedy is indeed competitive. 
    \item We then show which percentage of the states require a strictly non-greedy decision under the optimal policy, i.e., assigning to the furthest neighbour. While we know theoretically that this percentage is below 50\%, we observe that the actual numbers are very low. We remark that having 0\% in the case of 3 drivers occurs only because $M=8$, as for greater values of $M$ we do find strictly positive percentages.
\end{itemize}

\begin{table}[ht]
\centering
\begin{tabular}{l|ccccc}
\hline
\textbf{Drivers} & 3 & 4 & 5 & 6 & 7 \\
\hline
\% increased costs & 2.13 & 4.07 & 6.05 & 7.9 & 9.47 \\
\% non-greedy actions & 0 & 1.8 & 2 & 2.1 & 2.2 \\
\hline
\end{tabular}
\caption{Comparison of greedy and the optimal policy for $M=8$ and different number of drivers.}
\label{tab:GreedyVSOptimalM8}
\end{table}

To further understand these numbers, we now show in Table \ref{tab:N4MVarying} the results for $N=4$ and varying $M$. Crucially, we observe that greedy becomes more competitive as $M$ increases.

\begin{table}[ht]
\centering
\begin{tabular}{l|cccccc}
\hline
\textbf{$M$} & 5 & 6 & 7 & 8 & 9 & 10 \\
\hline
\% increased costs & 6.91 & 5.3 & 4.66 & 4.2 & 3.74 & 3.4 \\
\% non-greedy actions & 0 & 0 & 0.95 & 1.82 & 2.02 & 1.96 \\
\hline
\end{tabular}
\caption{Comparison of greedy and the optimal policy for $N=4$ and different values of $M$.}
\label{tab:N4MVarying}
\end{table}

Tables \ref{tab:GreedyVSOptimalM8} and \ref{tab:N4MVarying} reinforce that greedy is not so far from the optimum. We now show that even if we use the optimal policy, the IGD also appears. Let us first consider $N=5$, $M=9$, and analyse the temporal evolution of two indices, starting with drivers randomly located: (i) the cost and (ii) the variance of the intervals. We know that in the case of greedy these two quantities are correlated (Eq. \ref{Eq:CostEqualVar}), but this is not necessarily true for the optimal policy. Nevertheless, the variance still acts as a proxy for the IGD: If the IGD exists, then the longer intervals will become even longer, raising the variance. This is shown in Fig. \ref{fig:CostVarN5M9}, where we show the average of 10,000 repetitions of the experiment. We observe that:

\begin{figure}[ht]
    \centering
    \includegraphics[width=0.85\linewidth]{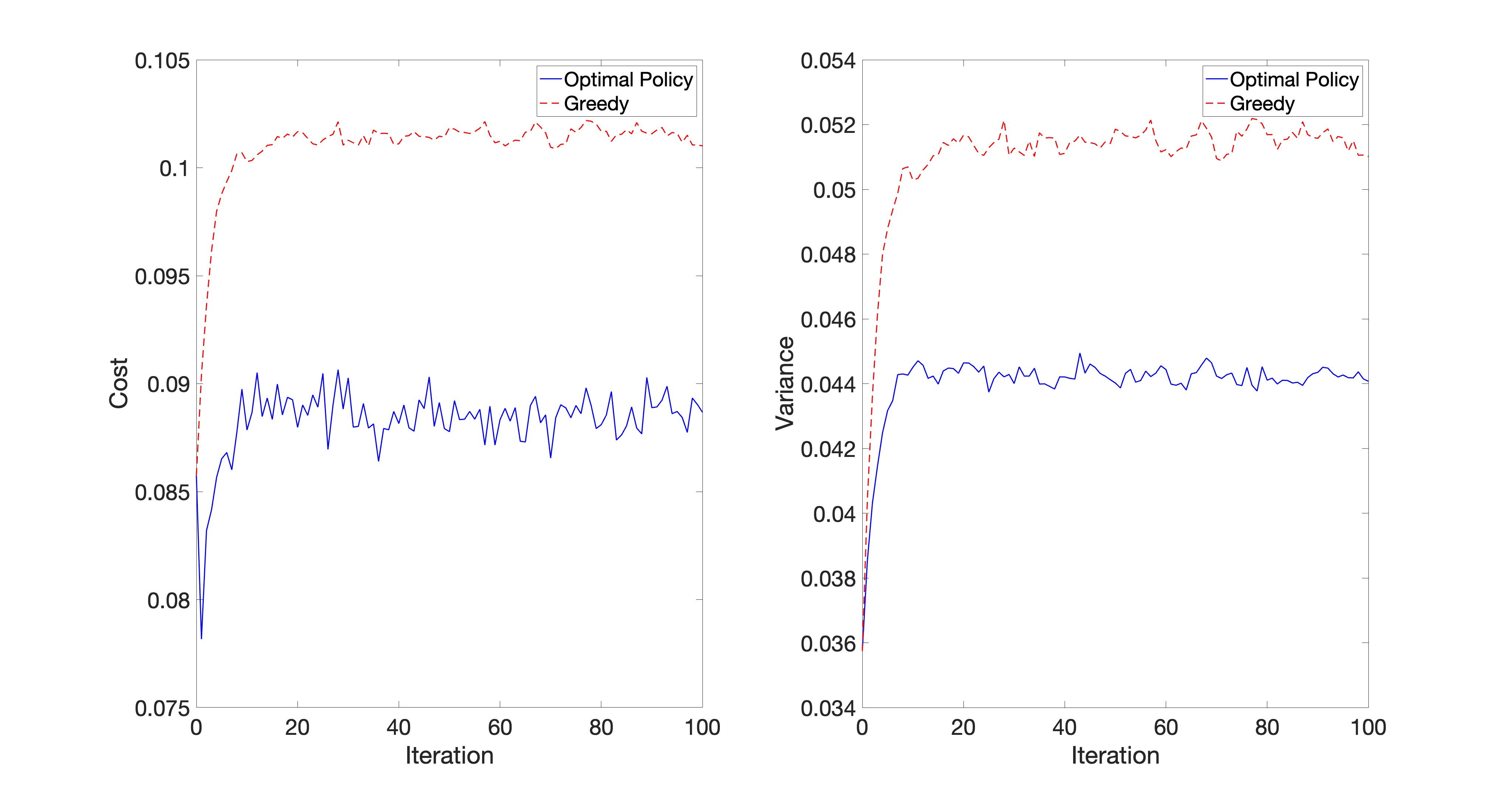}
    \caption{Comparison of the evolution of the system using the optimal policy versus greedy.}
    \label{fig:CostVarN5M9}
\end{figure}

\begin{itemize}
    
    \item While greedy is indeed worse than the optimal, both curves are also worse than the drivers being randomly distributed, represented by the starting point of both curves. That is, even when using an optimal policy the \textbf{IGD equilibrium is worse than a fully random distribution}.
    \item The reason for this increased costs stems from the variance. Even in the optimal policy, the variance increases significantly before converging.
    \item We can observe that the shape of the two red curves are almost identical, which is expected as the greedy costs depend directly on the variance. This is no longer true for the optimal policy, where there is some correlation but the curves are different.
    \item In the optimal policy, there is a clear drop at the beginning, after which the costs start to increase. This suggests that there is some smart way to respond to the initial random distribution of the drivers, but this smarter response does not persist for the subsequent iterations.
    
\end{itemize}

We now study the appearance of the IGD in a more direct way. As we have few drivers, there cannot be many long intervals. So we focus on the longest one. To be precise, we take the following steps:
\begin{enumerate}
    \item We consider a policy $\rho$ that can be optimal or greedy.
    \item For this policy, we compute the transition matrix between the positions of the drivers, i.e., $\Pi_{j,i}=P(s^t_1=i_1/M,\ldots,s^t_N=i_N/M|s^{t-1}_1=j_1/M,\ldots,s^{t-1}_N=j_N/M)$. Note that this can be done trivially because we know the assignment rule $\rho$.
    \item We compute the stationary probability $\pi$ of the matrix $\Pi$, and define $W_{a,b}=P(I(s^t)=b|I(s^{t-1}=a))$, where $I(s)$ is the longest gap between two consecutive drivers, when the drivers are described by $s$. We then compute $W_a=E_b(W_{a,b})$, where the expected value is taken over $b$ using the probability distribution $\pi$.
\end{enumerate}

In plain words, $W_a$ measures the expected longest gap in the next iteration, when the current longest gap is of length $a$. In Table \ref{tab:expected_max_gap} we report the results for $M=8,N=4$, where the ideal distribution is having a longest interval of $2/8$. To keep the Table simple we do not divide by 8. The main observations from this Table are:

\begin{table}[ht]
\centering
\begin{tabular}{l|ccccccc}
\hline
$a$ &2& 3 & 4 & 5 & 6 & 7 & 8 \\
\hline
$W_a$ Optimal & 3.5 & 3.77 & 4.18 & 4.71 & 5.2 & 5.65 & 6 \\
$W_a$ Greedy & 3.5 & 3.84 & 4.34 & 4.79 & 5.25 & 5.66 & 6 \\
\hline
\end{tabular}
\caption{Expected next max-gap given current max-gap $a$, for optimal and greedy policies, using $M=8$ (and omitting the division by 8).}
\label{tab:expected_max_gap}
\end{table}

\begin{itemize}
    \item The optimal policy is indeed better than greedy by keeping a better drivers' distribution. In every row, the optimal policy achieves a strictly lower value than greedy, with the exception of the two extreme cases where the assignment is trivial so the values are the same for both policies.
    \item However, the differences are quite mild. There is not much room from improvement compared to using the greedy assignment policy.
    \item The IGD can be observed when $W_a \geq a$, i.e., when the longest gap tends to further increase. Obviously, this cannot always hold, as in the extreme case where all drivers are in the same position, the longest gap will reduce. However, (i) This condition holds for exactly the same values of $a$ for both policies, and (ii) It even holds when $a=4$. That is, if all the drivers are concentrated in one half of the interval, the odds are that this gap will increase further in the next iteration, which constitutes a prototypical example of the IGD: it is just too probable that the user will arrive within this gap, and it will be assigned to one of the gap's borders regardless of the policy. This is not yet compensated by the arrival of the next driver.
\end{itemize}

In all, Table \ref{tab:expected_max_gap} shows that although the optimal policy does a better job than greedy in preparing for the future, this improvement is very minor and, more importantly, cannot prevent the IGD.

\subsection{Non-uniform arrivals}
So far, most of this section has assumed that passengers' and drivers' appearances follow a uniform distribution in $[0,1]$. However, in real transport systems there are usually hotspots where users and drivers tend to concentrate. In this subsection we discuss how this generalisation would affect our analysis and results. 

We first note that the overall dynamic remains almost the same. Instead of analysing the length of the intervals, we now need to focus on their probability. That is, an interval (gap) with a greater probability would still be more likely to attract the new users, making the interval to grow. Moreover, we know that Theorem \ref{Thm:NeihbourDrivers} remains valid. That is, even if we used an optimal policy, we would still assign one of the neighbours so that the gap increases. 

However, there is one relevant difference: the expected cost for a user arriving in an interval is determined by its length, whereas the probability of a user arriving there is governed by the interval's probability mass. In the uniform case, both quantities are proportional to the length, but this is not true when the probabilities are not uniform (i.e., Eq. \ref{Eq:CostEqualVar} is no longer valid in this general case). This introduces a random effect on the IGD: consider a situation where all the intervals have the same probability but some begin merging due to the IGD; if these were low-demand intervals, thus lengthy ones, the impact of the IGD is even worse, but the contrary would happen if merging short intervals.  

In all, the IGD still operates, but whether its impact is worse than when probabilities are uniform is unclear. We study this through simulations. We consider three different probability distributions. The first two give more probability to the interval $[0,1/2]$ (the high-demand zone) than to $[1/2,1]$ (the low-demand zone). The third one has a traditional triangular pdf. For each of them, we run Monte Carlo simulations, using 10 drivers and 100 iterations. The initial set of drivers is distributed randomly following the same corresponding probability distribution, so we can compare the final equilibrium with the random one. 

\begin{figure}[ht]
    \centering    \hspace{3.6cm}\includegraphics[width=0.7\linewidth]{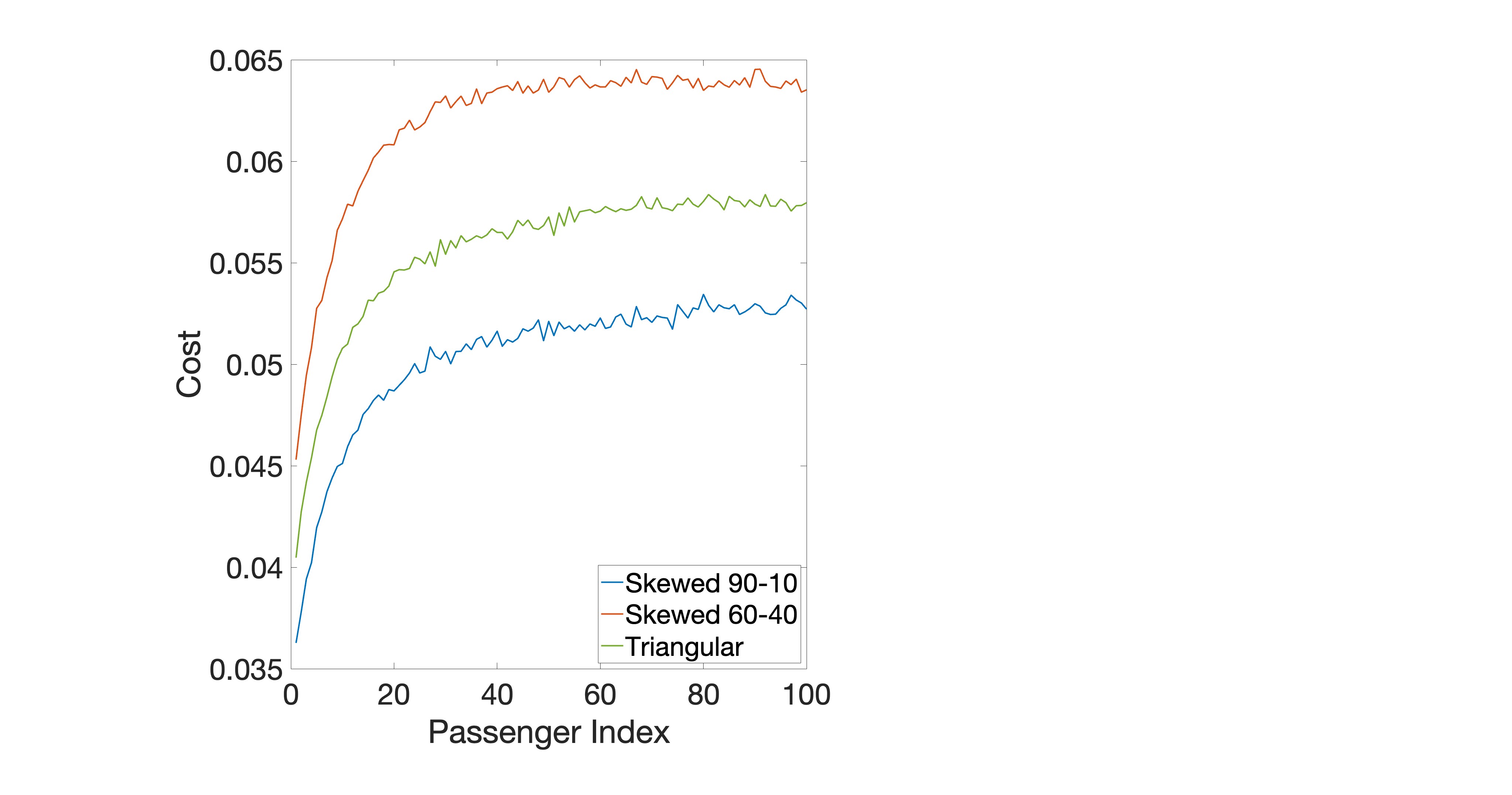}
    \caption{Effects of the IGD dynamics when the arrival of users and drivers is not uniform in space.}
    \label{fig:Comparison of distributions}
\end{figure}

The trends in Fig. \ref{fig:Comparison of distributions} clearly show that the IGD equilibrium is worse than the original setup with drivers randomly located. This is true for all three curves. Moreover, the comparison between the two skewed cases reveal that the more concentrated the drivers, the worse the impact of the IGD, as the relative increase in costs is significantly greater.

In all, the conclusions remain similar to what was discussed in the previous subsections. As a complement, let us remark that in section \ref{Sec:Manhattan} we run simulations using a real-world dataset from Manhattan, where the origins and destinations of the users are not uniform.

\section{Visualisations on the Unit Square}\label{sec:Square}

Let us consider now the (continuous) unit square in two dimensions. Conceptually, the system remains quite simple. But we now explain why an analytical treatment seems beyond reach. Note that what we described as the dominance zone in section \ref{sec:INTRO} corresponds exactly to the Voronoi diagram induced by the vehicles \citep{de2000computational}. This means that in the unit square the dominance zone of every vehicle is a convex polygon, and their geometry now plays a role. Crucially, the probability of assigning a given car is given by the area of its corresponding dominance zone, which is exactly the same as in the single-dimensional case (using the length instead of the area). But the expected cost is now quite different: if the area is regular (e.g. an hexagon or a square), the expected cost will be proportional to the polygon's side, i.e., to the square root of its area. But if the polygon is more irregular, its area becomes less informative. 

The complexity described in the previous paragraph hinders the development of a theoretical analysis. But on the other hand, the square is an ideal environment for visualisations, which can be very helpful for transport analysis \citep{andrienko2017visual}. Let us first remark that the overall dynamics remain quite similar. This can be explained by three key observations. First, it is still desirable to maintain dominance zones of comparable size. Second, larger zones have a higher probability of attracting users. Third, although a zone's area is no longer directly proportional to the expected user cost (as in the one-dimensional case), it does remain positively correlated, since larger zones tend to include points that are farther from the assigned vehicle. This is the version of the IGD dynamics in the square.

\begin{figure}[ht!]
    \centering
    \begin{subfigure}[b]{0.3\textwidth}
        \includegraphics[width=\textwidth]{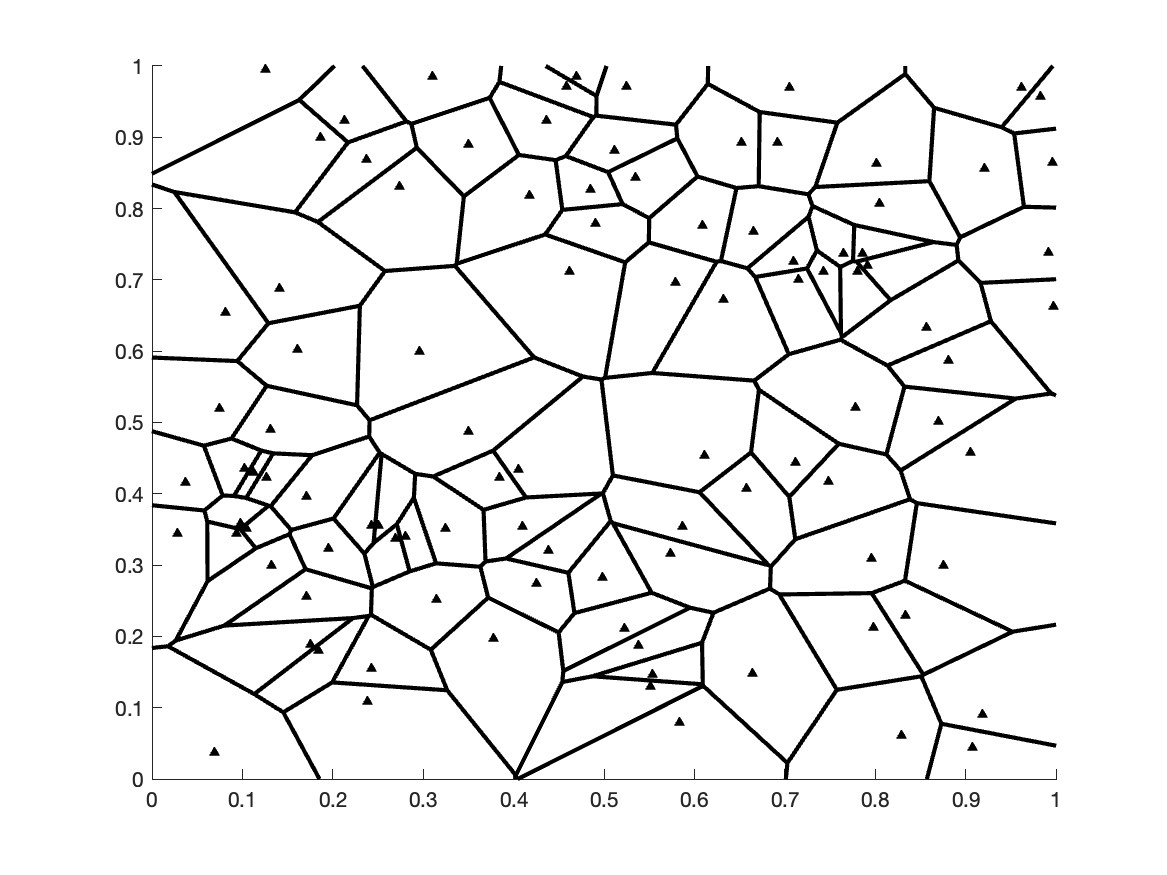}
        \caption{Random distribution}
        \label{fig:2DRandom}
    \end{subfigure}
    \begin{subfigure}[b]{0.3\textwidth}
        \includegraphics[width=\textwidth]{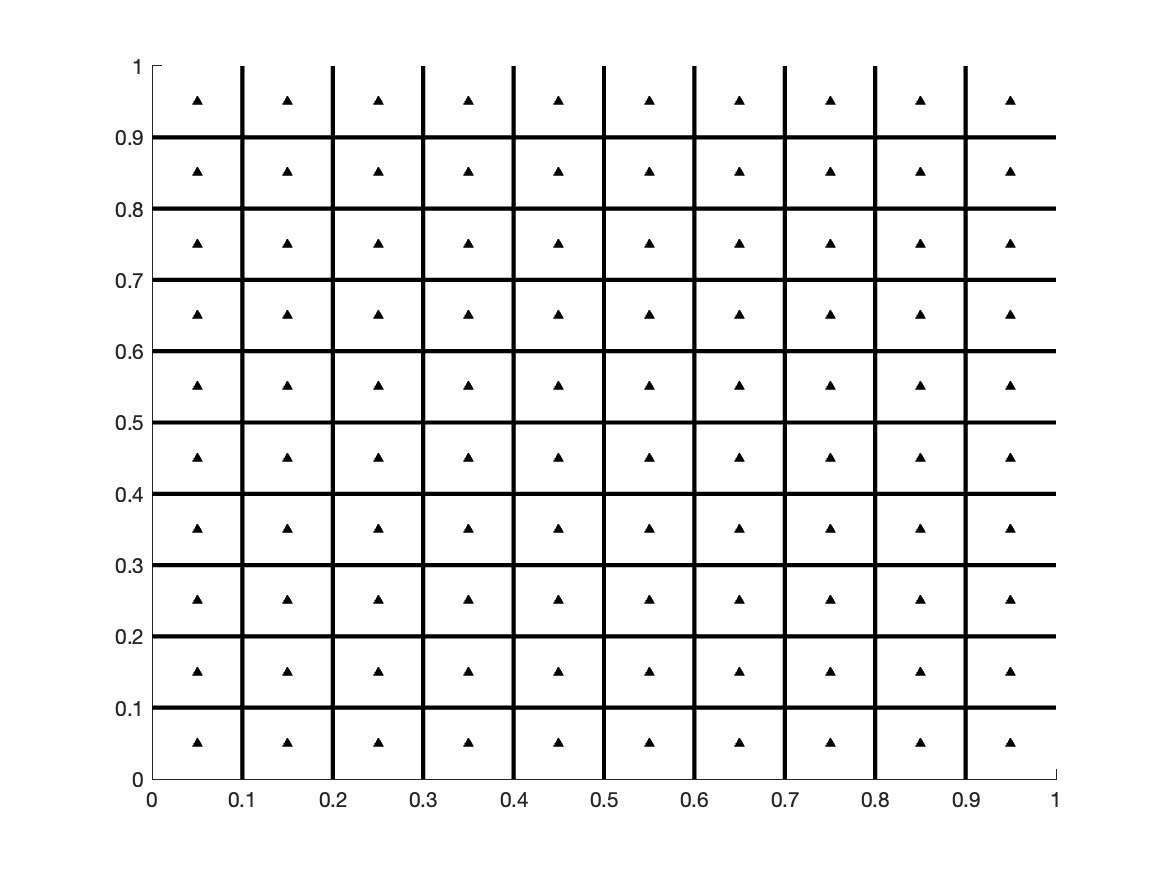}
        \caption{Initial distribution ($t=0$)}
        \label{fig:2DInitial}
    \end{subfigure}

    \begin{subfigure}[b]{0.3\textwidth}
        \includegraphics[width=\textwidth]{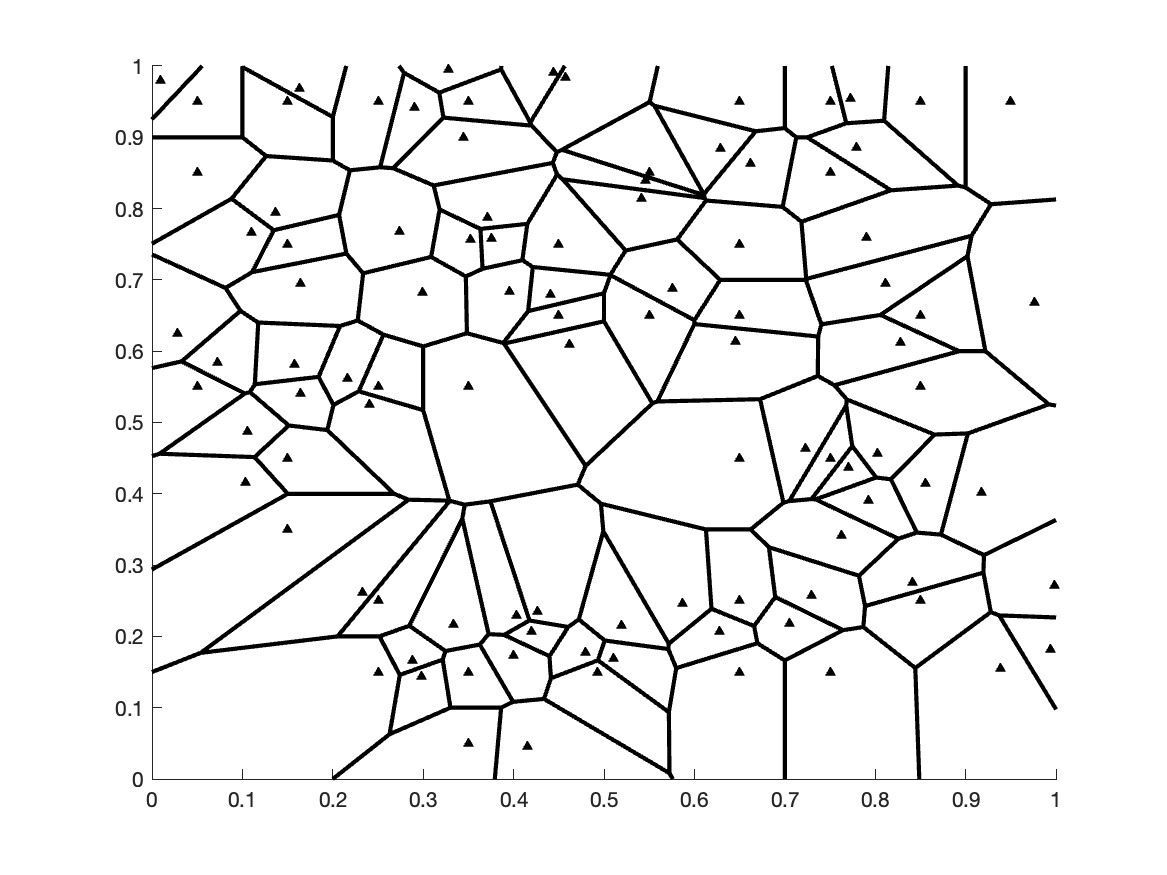}
        \caption{$t=100$}
        \label{fig:2D100}
    \end{subfigure}
    \hfill 
    \begin{subfigure}[b]{0.3\textwidth}
        \includegraphics[width=\textwidth]{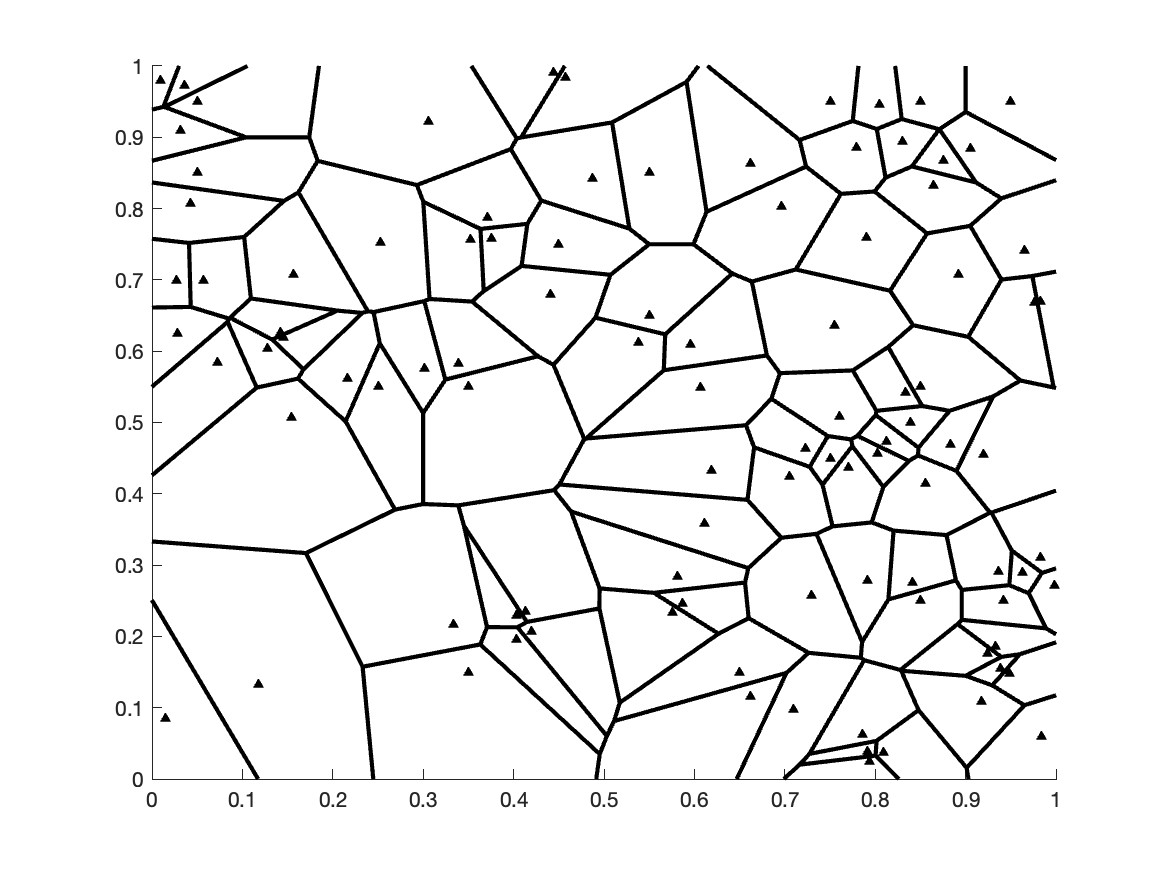}
        \caption{$t=200$}
        \label{fig:2D200}
    \end{subfigure}
    \hfill 
    \begin{subfigure}[b]{0.3\textwidth}
        \includegraphics[width=\textwidth]{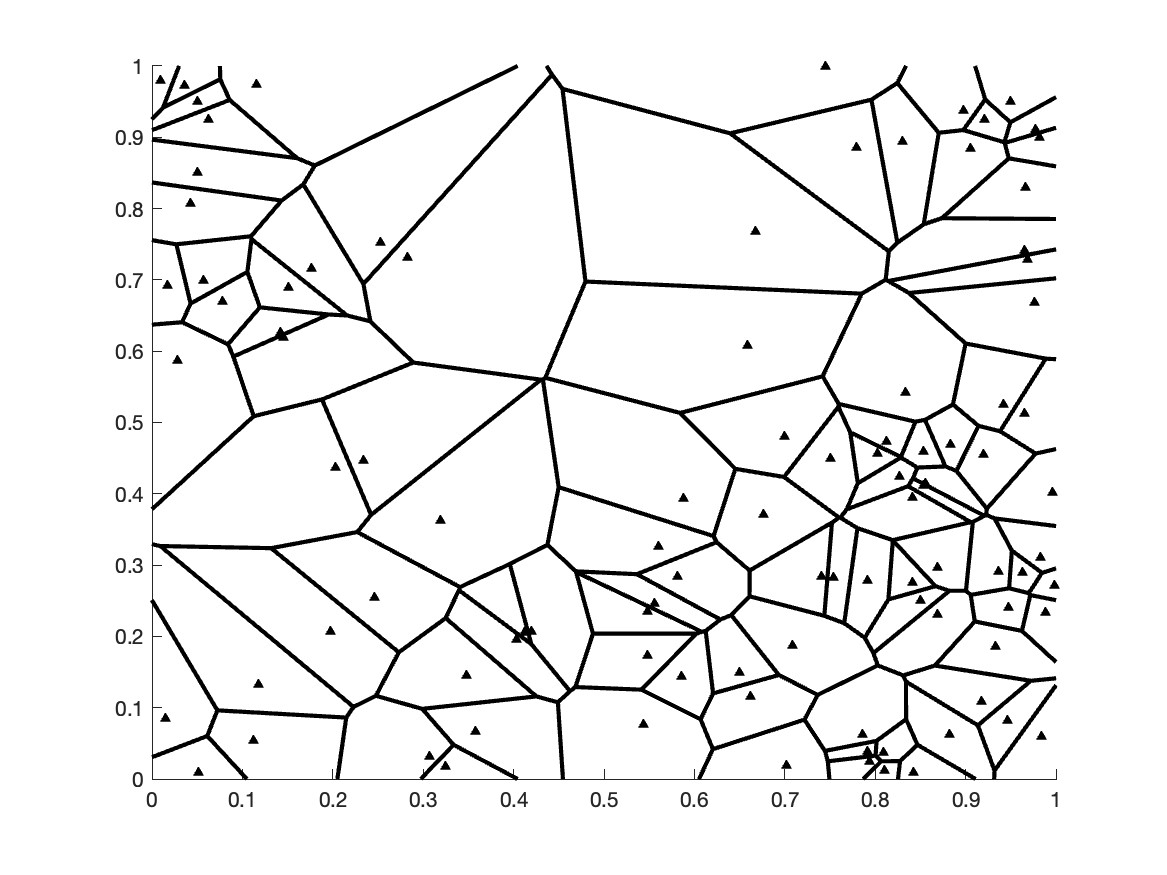}
        \caption{$t=300$}
        \label{fig:2D300}
    \end{subfigure}

    \caption{Evolution of the Markovian system in the unit square with 100 vehicles. Triangles represent the vehicles.}
    \label{fig:Voronoi}
\end{figure}

In order to develop the visualisations, we consider 100 vehicles that begin at the centres of a perfect partition of the square into 100 subsquares. We then run 300 iterations. The resulting process is shown in the video attached to this submission\footnote{The video can be accessed at youtu.be/2caWJjQyvXg}. We further visualise the results in Fig. \ref{fig:Voronoi}, where vehicles are represented as triangles. First, Fig. \ref{fig:2DRandom} shows a benchmark of sorts, namely one realisation of the Voronoi diagram with 100 points randomly distributed in the square. Figs. \ref{fig:2DInitial}-\ref{fig:2D300} show the evolution of the system, from the perfect initial to the distribution to the final one, showing the state of the system every 100 iterations. The main observations from Fig. \ref{fig:Voronoi} are:

\begin{itemize}
    \item The system's distribution keeps worsening over time. While already at iteration 100 the heterogeneity between the zones' areas is evident, this situation is significantly sharper in $t=200$ and $t=300$. Crucially, some very large zones start to appear.
    \item While the random distribution is also quite heterogeneous, its zones are clearly less distinct than those yielded by the IGD. Fig. \ref{fig:2DRandom} shows no huge zones and relatively few very small ones, compared to Fig. \ref{fig:2D300}.
\end{itemize}

This visual analysis can be quantified. To do this, in every iteration we compute the area of every dominance zone, and then calculate the variance of the resulting vector. The results are presented in Fig. \ref{fig:Variance100cars}, where we benchmark against the expected variance if the drivers are distributed randomly (computed through a Monte Carlo simulation). It becomes evident how the variance increases, and how it rapidly exceeds the random one. We note that other simulations might reach lower numbers, but in all of them we observe the same increasing trend and a much greater variance than the random case.

\begin{figure}[ht]
    \centering
    \includegraphics[width=0.4\linewidth]{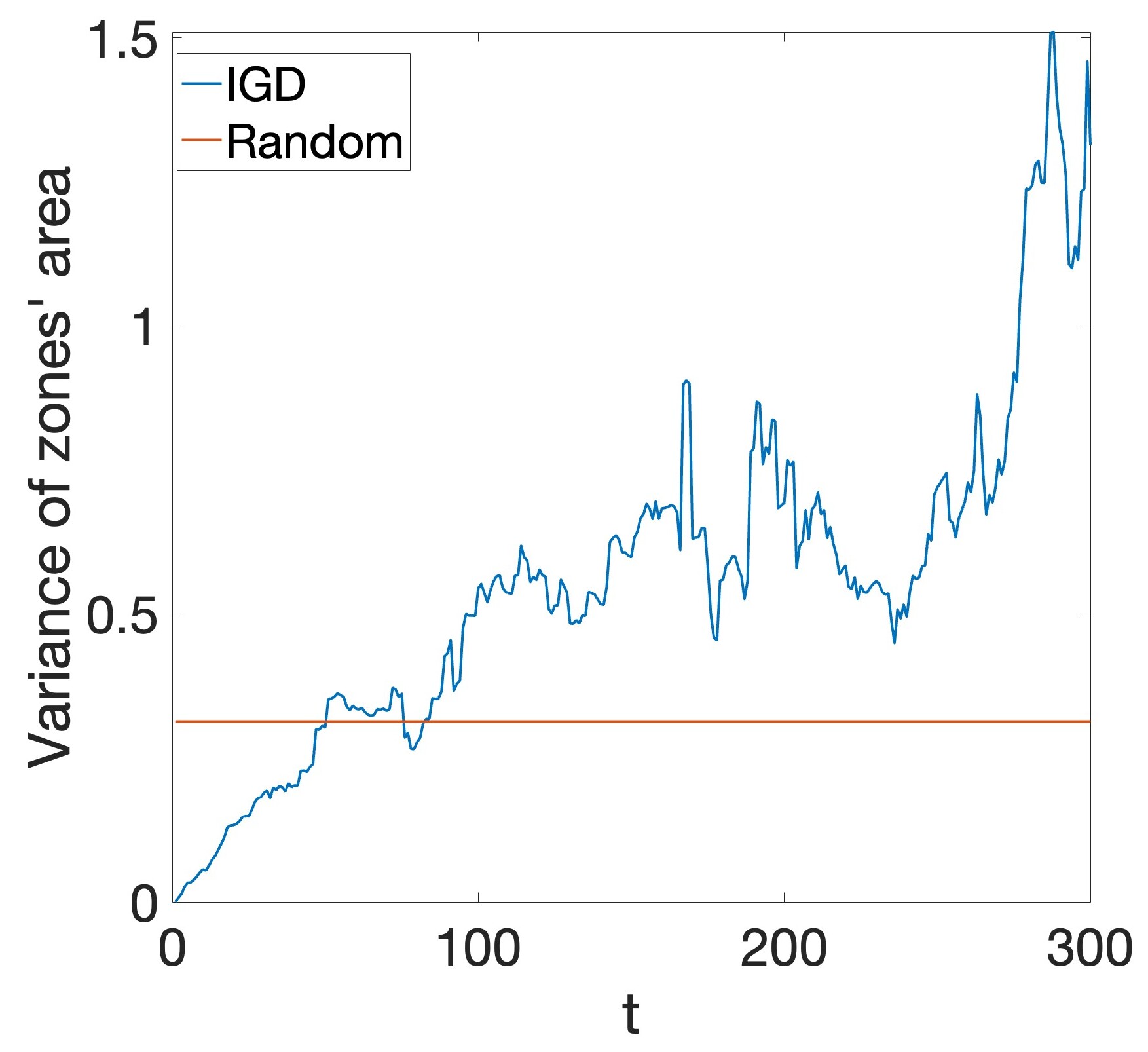}
    \caption{Evolution of the variance of the dominance zones in the unit square.}
    \label{fig:Variance100cars}
\end{figure}

It is worth noting that the sub-optimality of greedy assignments has been studied in a two-dimensional context by \cite{hyland2018dynamic}, who compare it to a matching method in which multiple users are assigned simultaneously (this is known as ``batch assignment'', further discussed and tested in section \ref{Sec:RingRadial}). They found that while greedy can be significantly outperformed, the performance difference narrows substantially as demand increases.

\section{Simulation of a Ride-hailing System and the Emergence of the WGC}\label{sec:Simulations}

The theoretical model permits a clear analysis, but how transferable is this analysis when going to more realistic scenarios? Can we still observe the dynamic described above, i.e., that the heterogeneity of the dominant areas creates a negative cycle? Moreover, when the number of idle vehicles is not constant, but depends on the actual service times, can we observe a relationship between the IGD and the WGC?

In order to study these questions, we run simulations of a ride-hailing system in two different settings. First, we use a ring-radial network (section~\ref{Sec:RingRadial}), where we can control and ensure some regularity in the network, demand, and supply, enabling us to analyse the IGD and WGC in better isolation. Next, we use a real-world dataset (Section~\ref{Sec:Manhattan}) to observe how these dynamics emerge when combined with the many additional aspects and complexities of real cities.

\subsection{Ring-radial network} \label{Sec:RingRadial}
We consider a ring radial network composed by 45 zones, each of them with $5\times 7$ nodes. This network has been previously utilised for the analysis of on-demand mobility, specifically to study scale economies in ride-pooling \citep{fielbaum2023economies}. The network, although with fewer zones, is illustrated in Fig. \ref{fig:ringradial}. All arcs are bidirectional and are traversed in half a minute. Requests are created randomly but keeping the distance between origin and destination within a narrow interval (between 5.8 and 6.2 minutes), in order to limit the relevance of randomness. We have one request every 30 seconds, and each request is assigned immediately upon appearing. Note that it is no longer true that every time a vehicle is assigned, another one becomes available, as in the theoretical model. Instead, a vehicle is assigned to a request, and it becomes again available after the drop-off. We consider 120 minutes of simulation, with part of the vehicles beginning with a user onboard as the result of a 30 minutes warm-up phase (that starts with one vehicle at the centre of each zone), and the remaining perfectly distributed in the network. The experiment is repeated 100 times and we report the average results in Fig. \ref{fig:simulationsUniform}.

\begin{figure}[ht]
    \centering
    \includegraphics[width=0.3\linewidth]{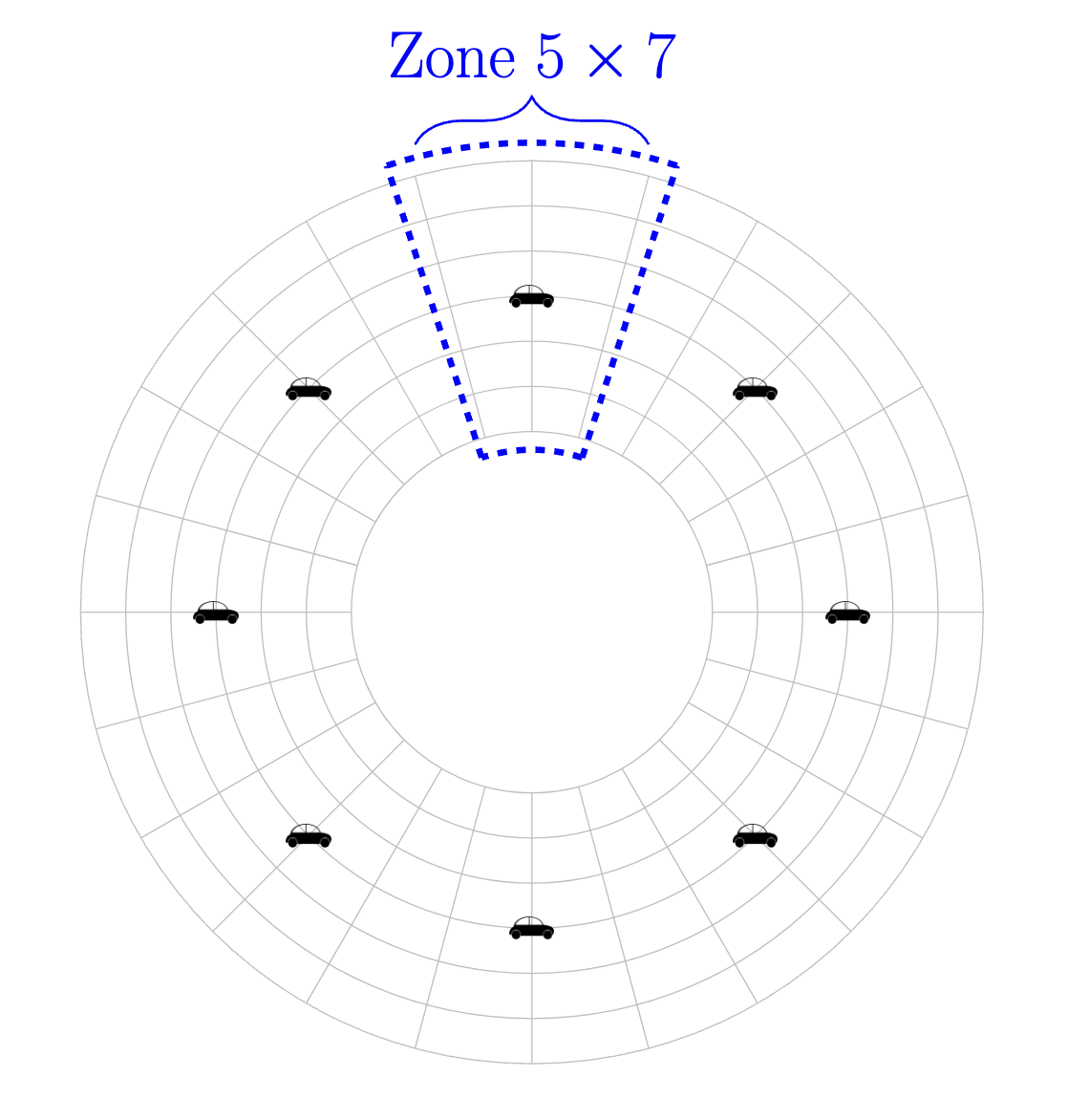}
    \caption{An illustration of the network, and the starting position of the vehicles, used for the simulations.}
    \label{fig:ringradial}
\end{figure}

Fig. \ref{fig:simulationsUniform} shows three panels. On the left, we show the average waiting time of the passengers, i.e., the exact cost function we have been considering. At the centre, we show the number of idle vehicles. Previous models that do not consider the evolution of the position of the vehicles have frequently assumed an inverted relationship between the two, e.g., $Waiting=Idle^{-\alpha}$ for some $\alpha>0$ \citep{yan2020dynamic}. Remarkably, at the beginning of our simulations (after a few minutes of adjustment), the waiting time goes up even though the number of idle vehicles also increases. The reason for this apparent contradiction lies precisely on the non-uniformity of the idle vehicles' position (i.e., the IGD), which is reported at the right and measured as follows: for every idle vehicle $v$, we compute the number of nodes that belong to the dominance zone of $v$, we calculate the variance of the resulting vector, and normalise.

\begin{figure}[ht]
    \centering
    \includegraphics[width=0.8\linewidth]{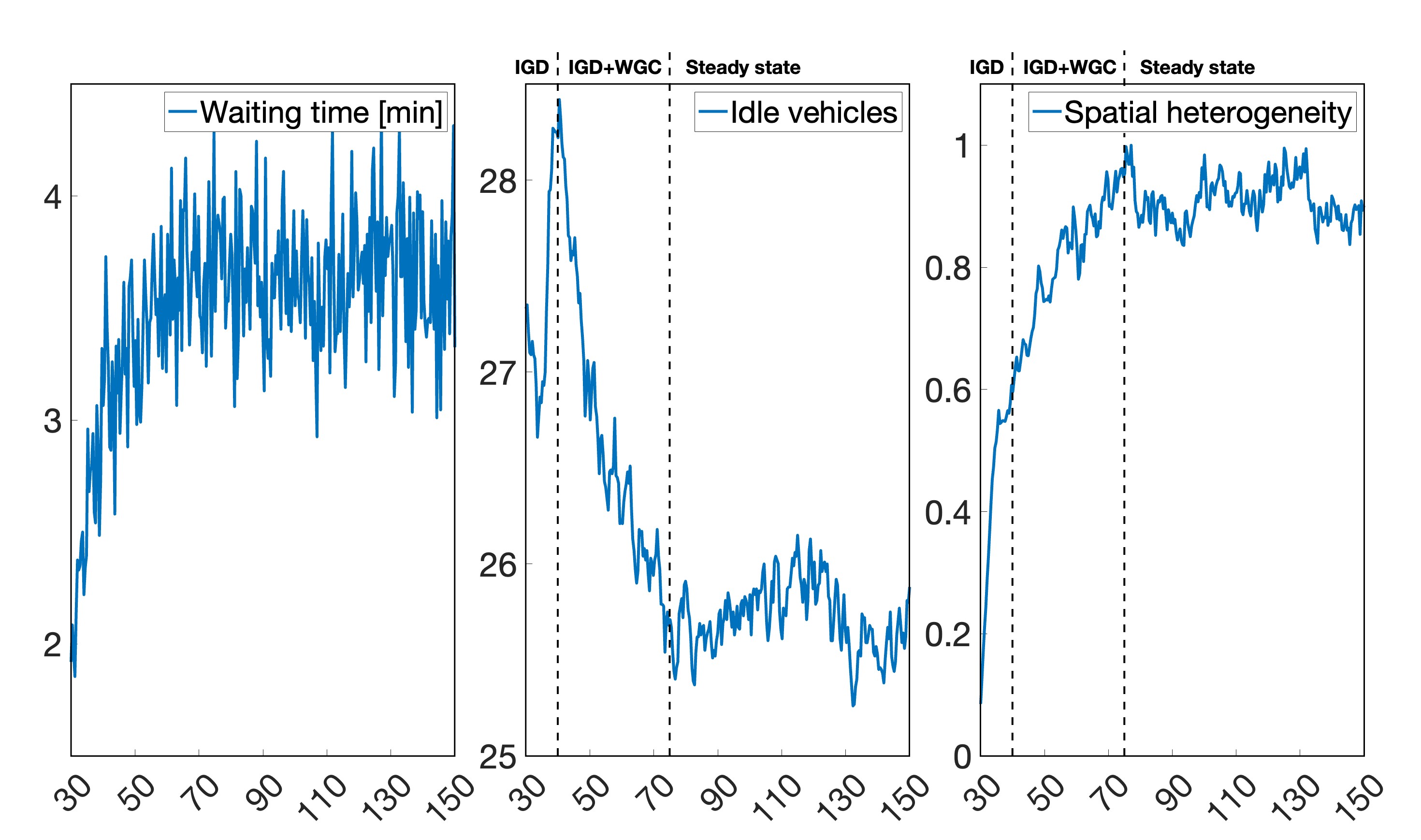}
    \caption{Simulation results on the circular grid.}
    \label{fig:simulationsUniform}
\end{figure}

We have divided the second and third panels of Fig. \ref{fig:simulationsUniform} into three distinct phases:

\begin{enumerate}
\item In the first phase, the number of idle vehicles increases, yet the waiting time also increases, indicating that the IGD is the dominant mechanism.
\item In the second phase, the number of idle vehicles decreases while spatial heterogeneity continues to increase. This means that both the IGD and the WGC are acting simultaneously.
\item In the final phase, the system converges to a steady state and the observed metrics stabilise.
\end{enumerate}

These results remain the same if we use \textit{batch assignment} \citep{yan2020dynamic,ramezani2023dynamic}. That is, instead of assigning a passenger each time it appears, we wait for some time to collect more requests, and then assign them all together. In this case, we assign every six minutes. The idea is that by collecting more information, the assignment decision will be better (less greedy). As our dynamic is related to assignments and how current decisions affect the future state of the system, it is worth analysing if assigning by batches might eliminate (or mitigate) the IGD. 

As shown in Fig. \ref{fig:simulationsBATCH}, the trends are consistent with those observed in Fig. \ref{fig:simulationsUniform}. The numbers do change compared to the case where passengers are assigned right away: Waiting times increase, as some passengers have to wait before being assigned, and the number of idle vehicles also increases, reflecting that the distance between users and their assigned drivers does decrease. However, the main feature from Fig. \ref{fig:simulationsUniform} also appears here: at the beginning of the simulations, the waiting time increases (Fig. \ref{fig:simulationsBATCH} left) even though the number of idle vehicles increases as well (Fig. \ref{fig:simulationsBATCH} centre), which is explained because the distribution of the vehicles gets worse (Fig. \ref{fig:simulationsBATCH} right).

\begin{figure}[h!]
    \centering
    \includegraphics[width=0.8\linewidth]{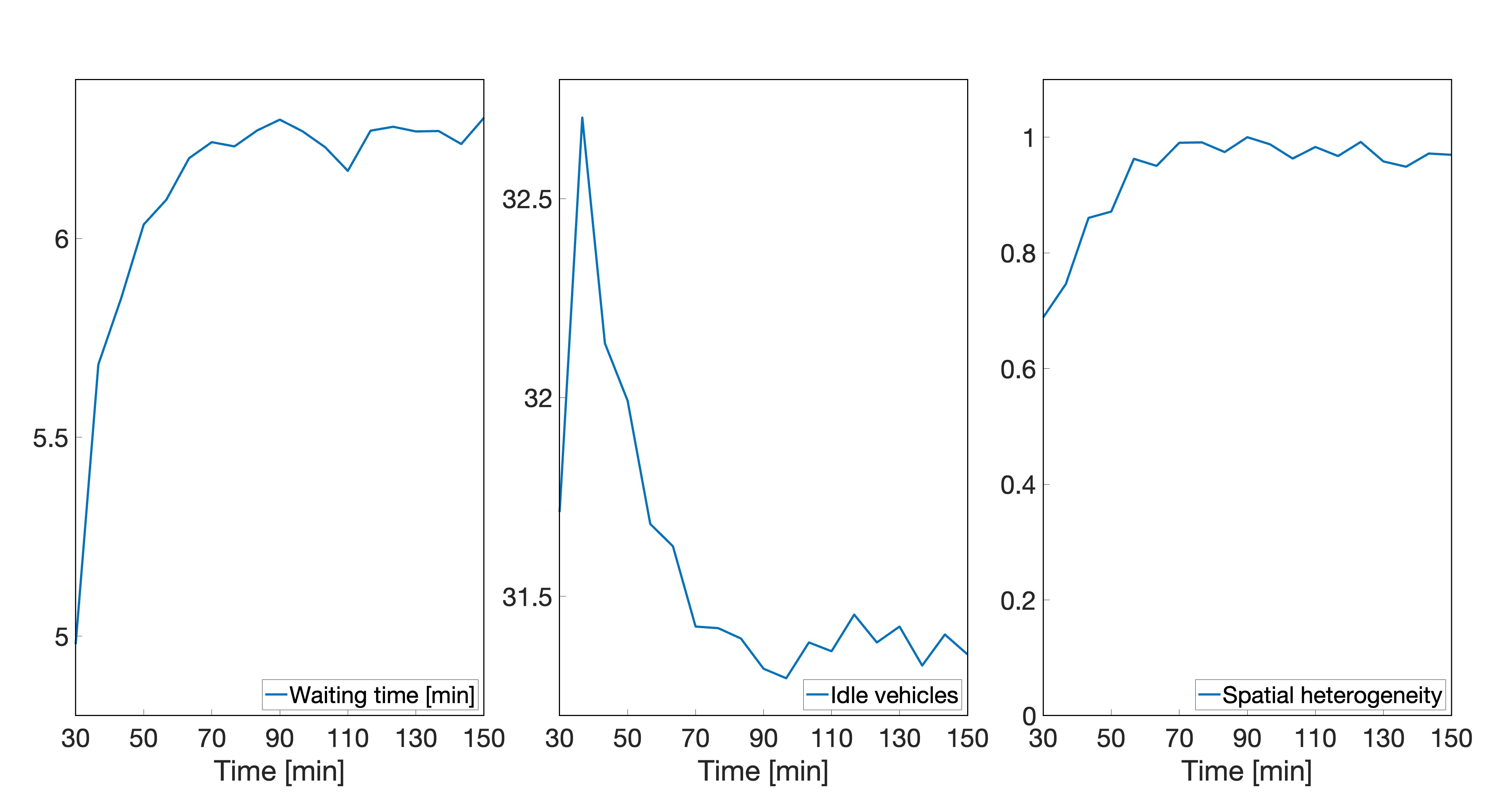}
    \caption{Results of the simulations on the circular grid when the assignmend is batch-based.}
    \label{fig:simulationsBATCH}
\end{figure}

\subsection{Simulations on a real-world dataset from Manhattan} \label{Sec:Manhattan}
We now analyse the emergence of the IGD in a realistic context, namely the widely used dataset from Manhattan. We consider a subnetwork of the Manhattan roadmap, first proposed by \cite{zhang2025walking}, to work over a more homogeneous context by removing the areas with very low demand. This network is composed by 1,966 nodes and 4,235 edges, as depicted in Fig. \ref{fig:ManhattanMap}.

\begin{figure}
    \centering
    \includegraphics[width=0.35\linewidth]{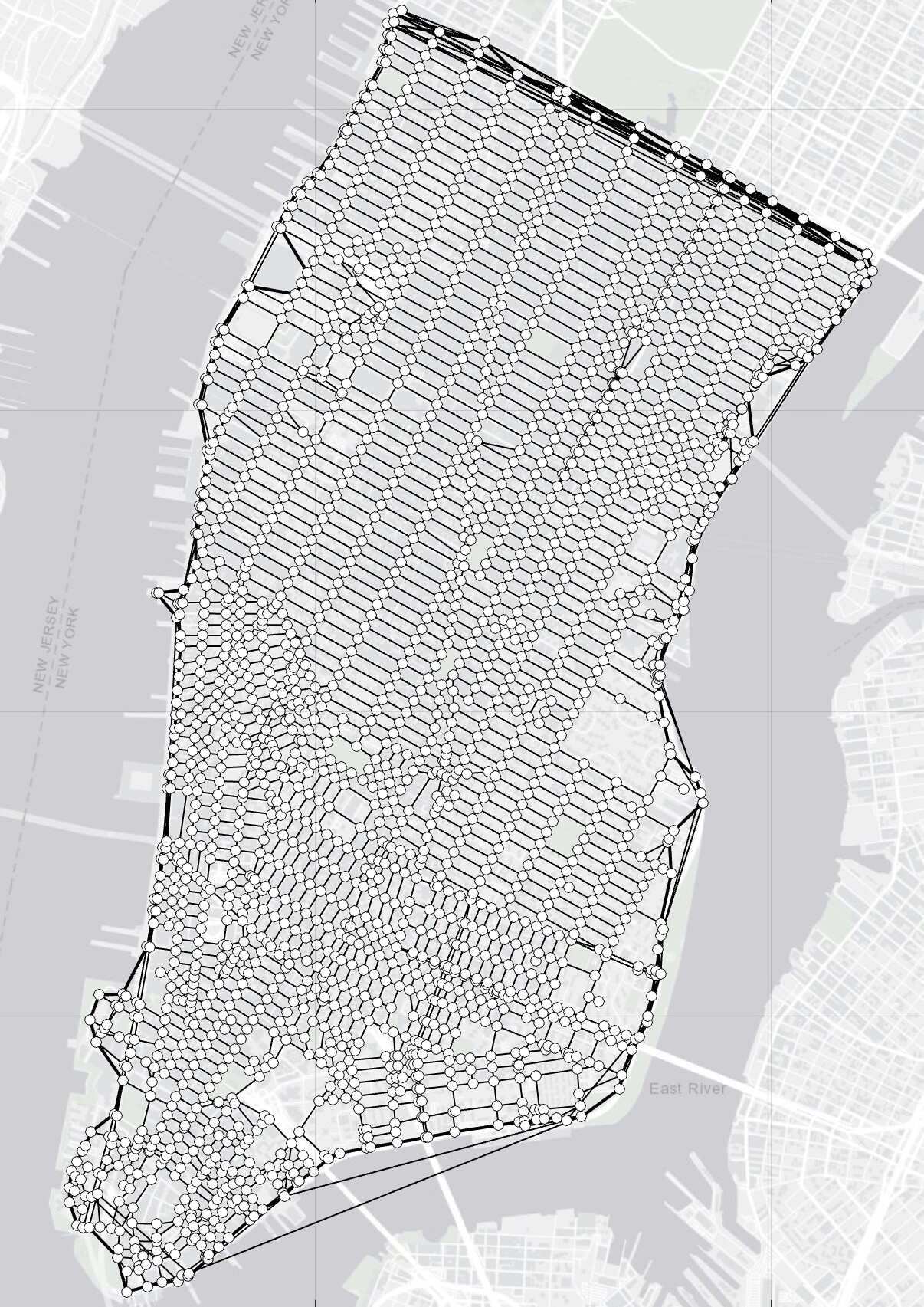}
    \caption{The subnetwork of Manhattan used in the simulations. Source: \cite{zhang2025walking}}
    \label{fig:ManhattanMap}
\end{figure}

We consider one hour of data from the morning peak, which comprises 14,213 requests, and a fleet of 1,500 vehicles\footnote{Formally, we redistribute these requests within 5 hours. This is done because we need a fleet that is large enough to ensure no rejections. This would require approximately 5,000 vehicles if not redistributing the requests, implying that almost every node would have a vehicle, and some several, which would hinder the analysis of the dominance zones.}. Requests are assigned as soon as they emerge. We implement a warm-up phase of 15 minutes, after which the idle vehicles begin their journey ``uniformly'' distributed: we use a k-medoids clustering method and situate the vehicles in the resulting centres\footnote{This clustering method is very appropriate for this context, as it can consider arbitrary distances and returns centres that are nodes in the network, as opposed to k-means that is based on the Euclidean distance and finds clusters but not explicit centres \citep{kaufman2009finding}.}.

Crucially, the demand is no longer spatially homogeneous, as some nodes might generate or attract more trips than others. Moreover, the length of every trip might differ. These changes hinder an analytical treatment, but test whether our theoretical results remain valid under this data-informed configuration. The fleet size was chosen to be large enough to ensure that all requests can be served, as otherwise, selecting which ones to reject could bias the results.

We compute the same set of results as for the previous experiments, shown in Fig. \ref{fig:ResultsManhattan}. In this case, the waiting times present a very large variation; moreover, it frequently reaches zero due to the high density of the vehicles in the network, that implies that sometimes a request would appear in a node that already has an idle vehicle. This is why we include a moving average, computed every 100 consecutive requests, to provide a cleaner interpretation. Crucially, we observe exactly the same trends as in the circular grid: the waiting time tends to increase (Fig. \ref{fig:ResultsManhattan} left) until converging. This is true even at the beginning of the simulations, where the number of idle vehicles increases as well (Fig. \ref{fig:ResultsManhattan}, disregarding the initial drop due to the inertia from the warm-up phase). The explanation lies in the IGD, as revealed by the increase in the spatial heterogeneity of the idle drivers (Fig. \ref{fig:ResultsManhattan} right).

\begin{figure}[ht]
    \centering
    \includegraphics[width=0.8\linewidth]{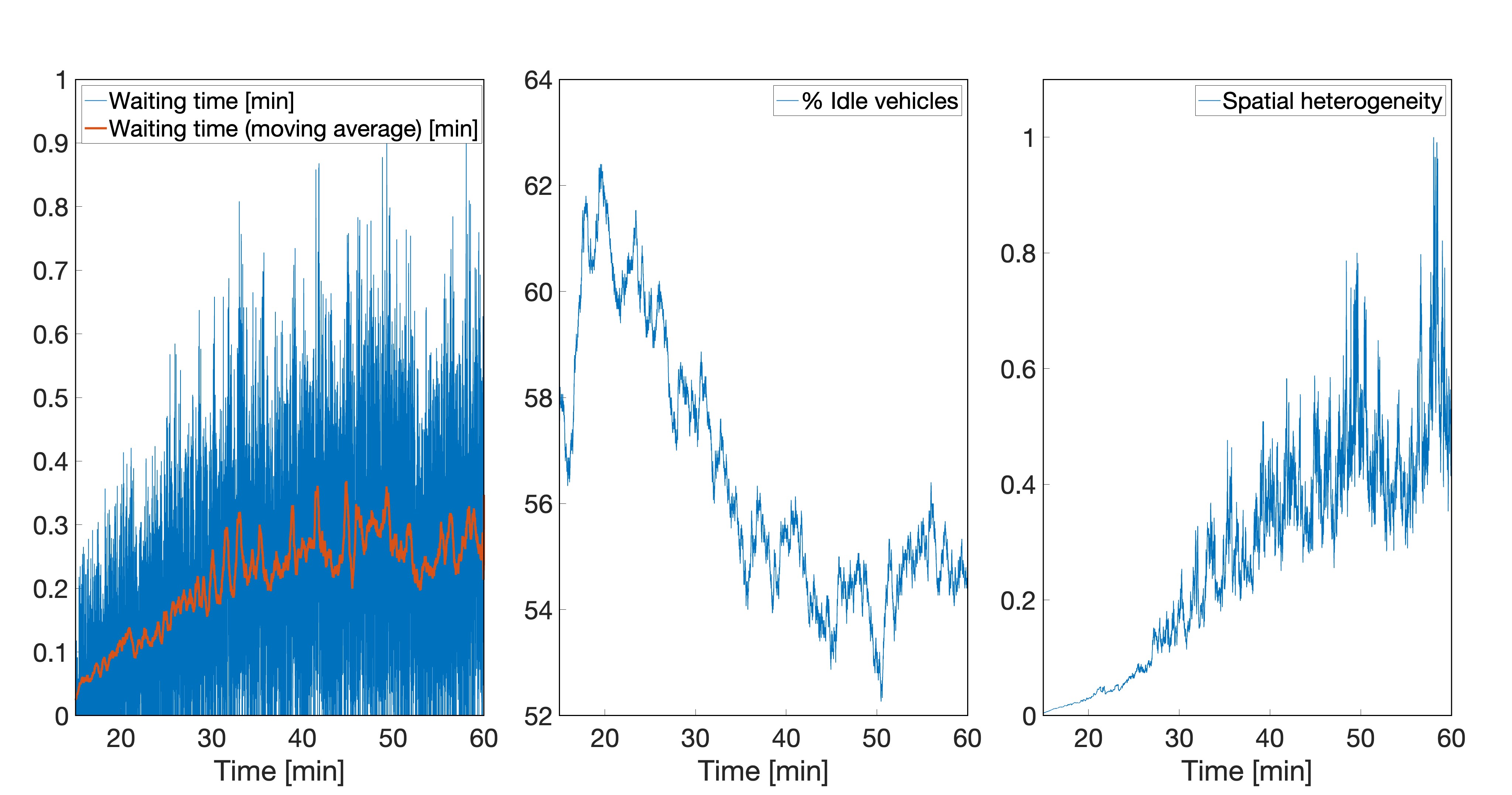}
    \caption{Results of the simulations on Manhattan. }
    \label{fig:ResultsManhattan}
\end{figure}

\section{Summary of the spatial dynamics under online matching}\label{Sec:SummaryDynamics}
Let us go back to the general spatial model, combining our theoretical findings with previous studies and the insights from the simulations. How do the IGD and the WGC relate and affect the spatial distribution of the servers? How does the distribution of the servers at equilibrium look like?

All of our analysis depict the following dynamics, illustrated in Fig. \ref{fig:SchemeIGD-WGC}. There are two feedback cycles, the IGD and the WGC. Both cycles are first triggered by the IGD. To be more specific, and assuming that the servers begin perfectly distributed:

\begin{enumerate}
    \item The random arrival of users and servers perturbs the original uniform distribution so that idle servers become unevenly distributed, which directly increases the average distance between users and servers.
    \item As explained by the IGD, the large dominance zones are more likely to attract users, so the servers become even less balanced - namely, because every user that appears in an already large dominance zone increases this dominance zone further.
    \item  The increased distance between users and servers implies that drivers take a longer time to become idle again, which reduces the number of idle servers, as described by \cite{castillo2017surge}.
    \item Having fewer idle servers obviously further increases the distance between a user and their assigned server. 
\end{enumerate}

\begin{figure}[ht]
\centering
  \makebox[\textwidth]{ 
    \resizebox{0.8\textwidth}{!}{ 
      \usetikzlibrary{arrows.meta, positioning}

\begin{tikzpicture}[
    state/.style={rectangle, rounded corners, draw=black, fill=green!20, text width=3.5cm, text centered, minimum height=1.2cm, font=\large},
    arrow/.style={draw, -{Latex[length=4mm,width=3mm]}, thick},
    label/.style={font=\normalsize, midway, fill=white, inner sep=1pt}
]
\hspace{-2cm}
\node[state] (s1) {Uniform distribution of idle servers};
\node[state, right=3cm of s1] (s2) {Idle servers become uneven};
\node[state, right=3cm of s2] (s3) {Increase in user-server distance};
\node[state, right=3cm of s3] (s4) {Fewer idle servers};

\draw[arrow] (s1) -- node[label, above] {Randomness} (s2);
\draw[arrow] (s2) -- node[label, above] {Direct effect} (s3);
\draw[arrow] (s4) -- node[label, above] {Direct effect} (s3);

\draw[arrow, bend right=45] (s3) to node[label, above] {Increasing Gaps Dynamics} (s2);
\draw[arrow, bend left=45] (s3) to node[label, above] {WGC Dynamics} (s4);

\end{tikzpicture} 
    }
  }
  \caption{Summary of the spatial dynamics in ride-hailing and similar systems.}
  \label{fig:SchemeIGD-WGC}
\end{figure}
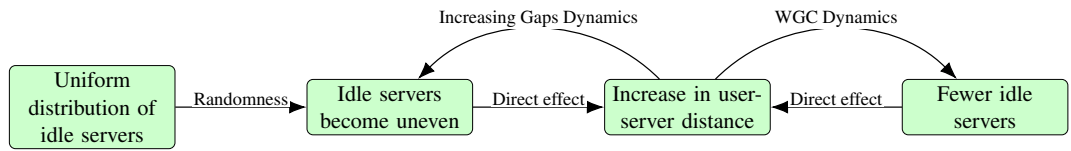

These feedback cycles degrade the system's quality of service until a new probabilistic equilibrium is reached. This equilibrium is not only worse than a servers' perfect distribution, but significantly worse than a fully random distribution.

As discussed in section \ref{sec:Greedy?} (where we discuss the suboptimality of the greedy policy), and reinforced in our experiments using batch assignment (section \ref{Sec:RingRadial}), there is no improvement to the matching algorithm that could substantially modify these dynamics. On the other hand, a proactive rebalancing of the idle servers could help maintain a better distribution. However, two caveats must be noted: first, under high-demand circumstances, the impact of rebalancing might be mild  because the percentage of idle servers remain low (\citealt{fielbaum2022anticipatory}); second, rebalancing is relatively simple in ride-hailing, but in other instances of spatial matching this is no longer the case, such as the examples given in the introduction, i.e., dockless bike-sharing and on-street parking when delivering.

It should be noted that WGC models (e.g. \citealt{castillo2017surge,li2025wild}) have typically identified that for a given fleet and demand rate, there are two possible equilibria, the ``good one'' with low pickup times and the WGC equilibrium with large pickup times. However, these models do not take into account how regions become spatially heterogenous, i.e., the IGD. What the dynamics summarised in Fig. \ref{fig:SchemeIGD-WGC} reveal is that, because of the IGD, pickup times will tend to increase from the ideal situation (the good equilibrium), and therefore the idle number of vehicles will diminish leading to the WGC equilibrium. This is also the interpretation of Theorem \ref{thm:CostsConverge}: there is a unique probabilistic limit in these processes.

\section{Conclusions}\label{sec:Conclusions}
Numerous emerging transport systems rely on some form of online matching between a user and the server (vehicle) it will utilise to travel.  Some traditional ones, including some examples of parking, are not done online but still follow a similar process. This \textit{spatial matching} is subject to particular dynamics that explain the temporal evolution of the spatial distribution of the servers. In this paper, we have identified and described the \textit{Increasing Gap Dynamics} (IGD), namely, that once the servers becomes unevenly distributed, those covering greater gaps are more likely to be assigned, further increasing those gaps. Through a number of analytical and experimental results, we have shown that the servers' distribution converges to a probabilistic equilibrium that is worse not only than a perfect distribution, but also than a fully random one.

Most of our theoretical analysis has been done over a circle and assuming a greedy assignment rule. Therein, we formally prove that a random distribution tends to become worse and that this happens exactly because of the IGD. Moreover, we have also proved that the optimal assignment rule will always match a user to one of its neighbouring servers, which suffices for the IGD to appear. We have also provided very evident visualisations of the phenomenon in a unit square. Simulations of a proper ride-hailing system using real-world data have shown not only that the IGD occurs, but also that it triggers the well-known \textit{Wild Goose Chase} (WGC).

As this is an emerging topic, there are numerous directions for further research. At the theoretical level, a formal bound on the relationship between the optimal and greedy algorithms on the circle remains an open problem. Similarly, studying other policies in more complex environments such as a graph or the unit square (e.g., the hierarchical greedy proposed by \citealt{kanoria2021dynamic}) could provide relevant insights into the evolution of the IGD in real-life scenarios. From an application perspective, while our results show that the IGD cannot be fully prevented, a crucial question emerges, namely, what remedial measures can be implemented to improve the efficiency and quality of service of the numerous systems that can be represented by this general spatial model. Finally, the models have been done for \textit{Assigned} modes, i.e., where users and vehicles are matched online (\citealt{fielbaum2025coordination}): would the IGD still appear for \textit{Onsite} modes, such as taxis, where the servers and users move randomly till finding each other?

\section*{Acknowledgments}
Dr Andres Fielbaum has been partially funded by the Australian Research Council (ARC) Discovery Early Career Researcher Award (DECRA) DE250100417. The research of Roberto Cominetti was partially supported by FONDECYT 1241805.
The idea of this paper was originally discussed at the Dagstuhl Seminar 24281 on Dynamic Traffic Models in Transportation Science.

\bibliographystyle{apalike} 
\bibliography{reference}

\section*{Appendix}

\subsection*{Lemma \ref{Lemma:Ergodic} and Theorem \ref{thm:CostsConverge} in the continuous case}

 In this section we establish the convergence of the greedy assignment for the stochastic dynamics in the case where drivers and users are distributed over a sufficiently regular continuous space $X$. 
Specifically, we consider  $X$ to be a compact subset of $\mathbb{R}^d$ with nonempty interior, endowed with its $d$-dimensional Lebesgue measure denoted by $\mu$ (i.e., the usual length if $d=1$, area if $d=2$, volume if $d=3$, etc.). We assume that $X$ is regular in the sense that for all $x\in X$ and $r>0$ we have
$\mu(X\cap B(x,r))>0$.
Since this measure varies continuously with $x$ over the compact set $X$,  for every fixed $r>0$ we can find $\alpha(r)>0$ such that 
$\mu(X\cap B(x,r))\geq\alpha(r)$
uniformly for all $x\in X$.
We also suppose that the probabilities
$P$ and $Q$ governing the appearance of users and drivers over the ground set $X$ have densities with respect to $\mu$, 
and that there exists $\rho>0$ such that for every Borel set $A\in\beta(X)$ we have $P(A)\geq\rho\mu(A)$ and $Q(A)\geq\rho\mu(A)$ (in particular this hold when the densities are continuous functions because $X$ is compact). 
We denote by $P_N(\cdot)$  and $Q_N(\cdot)$ the corresponding $N$-fold product probability measures on the space $X^N$,
and by $P(s,C)=\mathbb{P}(S^{t+1}\in C|S^t=s)$ the one-step transition kernel of the chain for each state $s\in X^N$ and every Borel set $C\in\beta(X^N)$. 
Further, we denote $D=\{s \in X^N: \exists i\neq j, s_i=s_j\}$ the exceptional set of $N$-tuples where two or more drivers coincide.

In order to extend Lemma \ref{Lemma:Ergodic} and Theorem \ref{thm:CostsConverge} to this continuous setting, we use Theorem 13.3.3 in \cite{meyn2012markov}, which provides sufficient conditions ensuring that the chain has a unique invariant probability measure $\pi(\cdot)$ and that 
for every initial state $s\in X^N$
the $t$-step 
distribution $\PP(S^t\in C|S^1=s)$ converges in total variation towards $\pi(C)$. Namely, this is guaranteed as long as the chain is:
\vspace{-1ex}

\begin{itemize}
\item[\bf a)] {\em $\varphi$-irreducible}: that is, there exists a non-zero measure $\varphi$ on $\beta(X^N)$ such that for every state $s \in X^N$ and each Borel set $C\in \beta(X^N)$ with  $\varphi(C)>0$ we have $\PP(S^t\in C|S^1=s)>0$ for some $t\geq 1$ (see Proposition 4.2.1\,(ii) in \citealt{meyn2012markov});\vspace{-1ex}

\item[\bf b)] {\em Harris recurrent}: that is, every Borel set $C\in\beta(X^N)$ with $\psi(C)>0$ satisfies $P_x(\Gamma_C=\infty)=1$, where $\Gamma_C$ is the number of times $C$ is visited and $\psi$ is the maximal measure  described in Proposition 4.2.2 in \cite{meyn2012markov};
\vspace{-1ex}
\item[\bf c)] {\em Aperiodic}: the maximum $m$ for which an $m$-cycle exists is $m=1$. An $m$-cycle is a partition of $X^N$ into $m$ subsets $C_1,\ldots,C_m$, such that for all $s \in C_i$ we have $ P(s, C_{i+1})=1$; and
\vspace{-1ex}
\item[\bf d)] {\em Positive}: there exists an invariant probability measure.
\vspace{-1ex}
\end{itemize}
In order to show that these conditions hold in our setting, we will exploit the following technical fact.

\begin{lemma}\label{Lemma:TechnicalAppendix}
    Let $C\in \beta(X^N)$ be any  Borel set with  positive Lebesgue measure $\mu_N(C)>0$. Then, there exists $\varepsilon>0$, which depends only on $C$, such that  $\PP(S^{2N+1}\in C | S^1=s) \geq \varepsilon$ uniformly for  $s \in X^N$.
    
\end{lemma}

\begin{proof}
We first argue that it suffices to consider the case where $C$ is a product $A_1 \times \ldots \times A_N$ of disjoint closed subsets $A_1,\ldots,A_N$ of $X$.
Indeed, by Lebesgue's density theorem almost every point in $C\setminus D$ has density 1 so that taking such a point $x \in C\setminus D$ and any $\kappa\in (0,1)$ we can find a small radius $r>0$ such that the $\|\cdot\|_\infty$-ball $R=\{s\in X^N:\|s-x\|_\infty\leq r\}$ satisfies 
$\mu_N(C \cap R) \geq \kappa\, \mu_N(R)$. Observe that $R$ is a product of rectangles $R=R_1 \times \ldots \times R_N$ and since $x\not\in D$ we can choose $r$ small enough so that the $R_k$'s are pairwise disjoint.
Setting $A_i=R_i\cap X$ it follows that the product set $\tilde C\triangleq A_1\times\ldots\times A_N$ 
is a Borel subset of $X^N\setminus D$, while
 the regularity of $X$ implies that $\mu_N(\tilde C)=\prod_{i=1}^N\mu(A_i)>0$.
Moreover, since $P_N$ and $Q_N$ have a strictly positive density it follows that  there exists some $\gamma>0$
such that
$\PP(S^{2N+1} \in C | S^1=s)\geq \gamma\; \PP(S^{2N+1} \in \tilde C | S^1=s)$. Hence, it suffices to find a uniform lower bound for the right hand side, that is,  for  sets  of product form $A_1\times\ldots\times A_N$
for pairwise disjoint $A_k$'s  with $\mu(A_k)>0$ and therefore $P(A_k)>0$ as well as $Q(A_k)>0$.

 Now, considering a potentially smaller radius $r>0$, we can further assume that there is a second family of pairwise disjoint closed subsets $B_1,\ldots,B_N$ included in the interior of $X$, which are disjoint from all the $A_k$'s and such that $\mu(B_k)>0$ (so that
$P(B_k)>0$ and $Q(B_k)>0$). Without loss of generality we may assume that the sets $B_k$ have a small diameter $\delta>0$ with $\delta$  smaller than the distance between any two of these sets $B_k$ and also smaller than the distance between these sets and  the boundary of $X$.
With this reduction we can assume from now on that $C$ is of the form $C=A_1 \times \ldots \times A_N$ with all these properties.
The rest of the proof consist in two steps: first, we show that starting from any initial state $s\in X^N$, after $N$ periods there is a uniformly positive probability that the chain reaches an intermediate state $S^{N+1}$ in which every $B_k$ contains exactly one driver. Then after $N$ additional periods we show that with a uniformly positive probability the chain reaches a state 
$S^{2N+1}\in A_1\times\ldots\times A_N$.

\noindent{\bf First step.} 
We claim that there exists $\varepsilon_1>0$, independent of the initial state $S^1=s$, such that with probability at least $\varepsilon_1$, after $N$ steps the drivers are relocated so that each $B_k$ contains exactly one driver.

Let $s\in X^N$ be an arbitrary initial state and let $\eta_k^1$ be the number of drivers in $B_k$. Denote $\eta^1_{N+1}=N-\sum_{k=1}^N\eta_k^1$ the drivers located in $X\setminus\cup_{k=1}^NB_k$. The superindex represents the time step $t=1$, as these values will evolve. We will show that there exists $\varepsilon_1'=\varepsilon_1'(\eta^1)$ that only depends on $\eta^1$ such that with probability at least
$\varepsilon_1'$ after $N$ steps
every set $B_k$ contains exactly one driver.
Since there is a finite number of possible configurations for $\eta^1$, by taking $\varepsilon_1$ as the minimum of these $\varepsilon_1'$ this property will hold uniformly for $s\in X^N$.  

In order to find  $\varepsilon_1'$ as required, we relocate the drivers one at a time  until we have $\eta^{N+1}_k=1$ for all $k=1,\ldots,N$ and $\eta^{N+1}_{N+1}=0$. In every iteration $t$, we have some sets $B_k$ that have exactly one driver. We call these the ``achieved'' sets with ``achieved'' drivers. When this happens for every $k=1,...,N$ we have 
completed our task. If this is not yet the case, we take any $j$ such that $\eta^t_j=0$. We want to remove a driver that is not  achieved, and replace it by a new driver located within $B_j$. We now prove that this occurs with a probability $\varepsilon_1''>0$ that 
only depends on the family of currently achieved sets $B_k$. Since there are finitely many combinations for such achieved sets, it then suffices to define $\varepsilon_1'$ as the product of all such $\varepsilon_1''$. 

Let $p=\min_{k}P(B_k)$ so that the probability that the new driver appears in any $B_j$ is at least $p$. Notice that this $p$ is strictly positive and does not depend on $s$ nor $\eta^1$. What about the probability of selecting a driver that is not achieved? It suffices to show that there exists $v_1>0$ such that the Lebesgue measure of the Voronoi regions covered by non-achieved drivers  is larger than $v_1$, since then the probability of the user being assigned to a non-achieved driver is at least $\rho\,v_1$ and we may then take $\varepsilon_1''=p\,\rho\,v_1$. In order to find an appropriate $v_1$ let us set $r=\delta/4$ and let $v$ be the $d$-volume of a half ball of radius $r$. We claim that we can take $v_1=\min\{\alpha(r),v\}$. In order to prove this, it suffices to 
 consider the case where there is a single non-achieved driver $s_o$ (if there are more, the joint volume of the Voronoi cells of non-achieved drivers would be larger). We distinguish two cases for $s_o$.
 \vspace{-2ex}
 
\begin{itemize}
    \item \textbf{Case 1:} If for every achieved $B_k$ we have $\mathop{\rm dist}(s_o,B_k)\geq\delta/2$, then  $B(s_o,r)\cap X$ is included in the Voronoi cell of $s_o$ which then has  a Lebesgue measure  at least $\alpha(r)\geq v_1$.
     \vspace{-1ex}

    \item \textbf{Case 2:} There is some achieved $B_k$ such that $\mathop{\rm dist}(s_o,B_k)<\delta/2$. Note that the choice of $\delta$ implies that this $k$ is unique and also that 
    $B(s_o,r)
    \subseteq X$.
    Let $s_k$ be the position of the only driver in $B_k$.  The Voronoi partition that considers just  $s_o$ and $s_k$ is given by a hyperplane passing through the mid point of the segment connecting $s_k$ and $s_o$. Now, at least half of the ball 
    $B(s_o,r)$
    is included in the half-space containing $s_o$, and since all the other achieved sets $B_i$ are at a distance larger than $\delta/2$ from $s_o$, this half ball is  included in the Voronoi cell of $s_o$ which then has a Lebesgue measure  of at least $v\geq v_1$.
\end{itemize}
 \vspace{-2ex}

\noindent Both cases combined show that regardless of the position of $s_o$, the described $\varepsilon''_1$ works. This concludes the first step: there is a lower bound $\varepsilon_1$, independent of $s$, such that $P(\eta^{N+1}_1=1,\ldots,\eta^{N+1}_N=1|S^1=s)\geq\varepsilon_1$.

\vspace{1ex}
\noindent{\bf Second step.} We now prove that the probability of moving from an arbitrary state $s$  having exactly one driver in each $B_k$ to a state $\tilde s\in A_1\times\ldots\times A_N$, also has a uniform lower bound $\varepsilon_2$ independent of $s$. 
Let $q>0$ be such that $Q(B_k)\geq q$ and $P(A_k)\geq q$ for all $k=1,\ldots,N$. Consider the driver at position $s_1$ and let  $B_{j}$ the set that contains $s_1$. Then, with probability $Q(B_{j})>q$ an  arriving user will appear in this set $B_j$ and will be assigned to the driver $s_1$, while with probability $P(A_1)\geq q$ the new driver will appear in a position $\tilde s_1\in A_1$.
Hence, with probability $q^2$ the state $s$ will transition to a new state with the first driver located in $A_1$. Proceeding sequentially by replacing the drivers at $s_2,s_3,\ldots$ with new drivers at $\tilde s_2\in A_2$, $\tilde s_3\in A_3, \ldots$
it follows that  after $N$ periods the state $s$ will transition to a state in $A_1\times\ldots\times A_N$, with probability at least $q^{2N}$.
\\[2ex]
Combining the first and second steps it follows that after $2N$ periods the state $S^{2N+1}$ has a probability at least $\varepsilon=\varepsilon_1 \varepsilon_2$ to belong to the target set $C=A_1\times\ldots\times A_N$, uniformly in the initial state $S^1=s$.
\end{proof}

\begin{corollary}
\label{Cor:TechnicalAppendix}
    Let $C\in \beta(X^N)$ be any  Borel set with  positive Lebesgue measure $\mu_N(C)>0$. Then, there exists $\varepsilon>0$, which depends only on $C$, such that  $\PP(S^{t}\!\in C) \geq \varepsilon$  for  all $t\ge 2N+1$, regardless of the starting state $S^1$.
\end{corollary}
\begin{proof}
 By conditioning on the state $S^{t-2N}$ visited at stage $t-2N$ we have
$\mathbb{P}(S^t\!\in C)=\mathbb{E}\big[\mathbb{P}(S^t\!\in C|S^{t-2N})\big]$. Since the
chain is homogeneous, this is equal to
$\mathbb{E}\big[\mathbb{P}(S^{2N+1}\!\in C|S^{1})\big]$ and we conclude by  Lemma \ref{Lemma:TechnicalAppendix}.
\end{proof}
Using these results
we may now proceed to establish the conditions 
that ensure the existence of a unique invariant measure and the convergence in total variation of the $t$-step distribution of the chain.
\\[1ex]
{\bf Condition a).} 
In order to check
$\varphi$-irreducibility we may just take
$\varphi=\mu_N$ since  Lemma \ref{Lemma:TechnicalAppendix} ensures that for  every initial point $s \in X^N$ and any Borel set $C$ with strictly positive Lebesgue measure $\mu_N(C)>0$  the chain has a strictly positive probability of reaching $C$ at stage $t=2N+1$.
\\[1ex]
{\bf Condition b).} 
  Lemma \ref{Lemma:TechnicalAppendix} implies that $\mathbb{P}(\Gamma_C=\infty)=1$ for all $C$ with positive Lebesgue measure $\mu_N(C)>0$. However, Harris recurrence needs this to be true for any set with $\psi(C)>0$ where $\psi$ is a maximal measure as defined in Proposition 4.2.2 in \cite{meyn2012markov}. Now,  since the chain is also irreducible with respect to $\psi$ (as stated in the  Proposition 4.2.2 just cited), every set $C$ with $\psi(C)>0$ satisfies $\mu_N(C)>0$. Indeed, if $\mu_N(C)=0$, and since the probabilities $P_N$ and $Q_N$ have densities with respect to $\mu_N$,  the probability for the drivers to reach $C$ would be zero for every $t$, implying that $\psi(C)=0$. Hence, Lemma \ref{Lemma:TechnicalAppendix} is indeed sufficient to ensure that the chain is Harris recurrent.
\\[1ex]
{\bf Condition c).} In any $m$-cycle $C_1,\ldots, C_m$ the sets $C_i$ are only visited at regular intervals of length $m$. Now, since $P_N$ and $Q_N$ are absolutely continuous with respect to the Lebesgue measure we have $\mu_N(C_i)>0$ and then Corollary \ref{Cor:TechnicalAppendix} shows that $C_i$ can be reached with positive probability at every stage $t\geq 2N+1$. This  necessarily implies that $m=1$.
\\[1ex]
{\bf Condition d).} To prove the existence of an invariant probability we use Theorem 3.1 in \cite{lasserre2000invariant}.  
Since the set $D=\{s \in X^N: \exists i\neq j, s_i=s_j\}$ is such that $\mu_N(D)=0$, and $P$ as well as $Q$ have a density with respect to the Lebesgue measure, the probability that the chain jumps in one step from a state $s\in X^N\setminus D$ towards $D$ is zero, while Lemma \ref{Lemma:TechnicalAppendix}
gives $\mathbb{P}(S^{2N+1}\not\in D|S^1=s)>0$ for all $s\in D$. On the other hand, we claim that the weak Feller property holds at every point $s\not\in D$, namely, for every continuous function $f\in\mathcal{C}(X^N)$ and each $s\in X^N$, by considering the Voronoi cells
$V^i(s)=\{x\in X:\|x-s_i\|_2\leq \|x-s_j\|_2\mbox{ for all }j\neq i\}$ for the Euclidean norm $\|\cdot\|_2$, we have $$\mathbb{P}f(s)\triangleq \int_{X^N} f(x) \,P(s,dx)=\sum_{i=1}^N Q(V^i(s))\int_X f(s_1,\ldots,s_{i-1},x_i,s_{i+1},\ldots,s_N)\,dP(x_i).$$
Below we prove that for each $i=1,\ldots,N$ the function $s\mapsto Q(V^i(s))$ is continuous at every $s\in X^N\setminus D$, so that the same holds for the 
map $s\mapsto \mathbb{P}f(s)$.
Finally, take  $f_0:X^N\to [0,1]$ a continuous function that vanishes exactly on $D$, that is $\{s\in X^N:f_0(s)=0\}=D$.
Take $C$ a compact subset of $X^N\setminus D$ with $\mu_N(C)>0$, and let $m=\min_{s\in C}f_0(s)$ so that $m>0$.
For an arbitrary initial state $s_0\in X^N$ we have $\mathbb{E}_{s_0}[f_0(S^t)]\geq m\,\mathbb{P}(S^t\in C)$ and then Corollary \ref{Cor:TechnicalAppendix} implies 
$$\limsup_{n\to\infty} \mbox{ $\mathbb{E}_{s_0}[\frac{1}{n}\sum_{t=0}^{n-1}f_0(S^t)]$}\geq m\,\varepsilon>0.$$
These conditions allow us to invoke Theorem 3.1 in \cite{lasserre2000invariant}
which ensures that the chain has an invariant probability measure $\pi$, establishing condition d).

\vspace{1ex}
\noindent{\bf Continuity of $s\mapsto Q(V_i(s))$.} Consider two $N$-tuples of drivers $s,r \in X^N\setminus D$ and let  $\delta=\|s-r\|$. Consider a point $x \in X$ such that $x \in V_i(s)\setminus V_i(r)$, that is, a user at $x$ would be assigned to the $i$-th driver when the drivers are located according to $s$ but to a different driver $j\neq i$ under $r$, so that
$\|x - r_j\|\leq\|x - r_i\|$ while 
$$\|x - r_i\| \leq \|x - s_i\| + \|s_i - r_i\|\leq\|x - s_j\| + \delta \leq\|x - r_j\| + \|r_j - s_j\| + \delta \leq\|x - r_j\| + 2 \delta.
$$
It follows that $V_i(s)\setminus V_i(r)\subseteq\cup_{j\neq i}\,\Delta_{ij}^\delta(r)$ where
$\Delta_{ij}^\delta(r)\triangleq \{x\in X:0\leq \|x-r_i\|-\|x-r_j\|\leq 2\delta\}$. These regions decrease with $\delta$ and, because $r_j\neq r_i$, for $\delta=0$ they become the intersection of $X$ with a hyperplane through the midpoint between $r_i$ and $r_j$.
Since $X$ is bounded, it follows that $\mu(\Delta_{ij}^\delta(r))\to 0$ as $\delta\to 0$. Then, 
the Lebesgue measure of the symmetric difference $V_i(s)\!\vartriangle\! V_i(r)$ of the Voronoi cells, and {\em a fortiori} the probabilities $Q(V_i(s)\!\vartriangle\! V_i(r))$, tend to 0 
when $\delta$ tends to 0.
Since $|Q(V_i(r))-Q(V_i(s))|\leq Q(V_i(s)\!\vartriangle \! V_i(r))$, we conclude that for any fixed tuple $s\in X^N\setminus D$ and an arbitrary sequence $r_n\to s$, we have $Q(V_i(r_n))\to Q(V_i(s))$.

 \paragraph{Conclusion:} Having proved that the conditions a)-d) hold  for any continuous ground set $X$ and probabilities $P$ and $Q$ that satisfy the stated regularity assumptions, Theorem 13.3.3 in \cite{meyn2012markov} implies the convergence of the chain in total variation  towards the unique invariant measure $\pi$. It follows that Lemma \ref{Lemma:Ergodic}
and Theorem \ref{thm:CostsConverge} in section 3 remain valid under these conditions.
\qed

\subsection*{Proof of Lemma \ref{Lemma:ChangeInIV}}

\begin{proof}
We compute the expected value $\Delta V(\ell)$ by conditioning on which intervals are merged and where the new driver falls. Namely, denote $M_k$ the event in which the new user appears in a position that produces the merging of $I_k$ and $I_{k+1}$, which has probability $P(M_k)=(\ell_k+\ell_{k+1})/2$, and let  $D_j$ the event where the new driver appears in the interval $I_j$ so that $P(D_j)=\ell_j$. We distinguish the cases
$j\in\{k,k+1\}$ and $j\not\in \{k,k+1\}$, by considering separately the events $A_k=M_k\cap(D_k\cup D_{k+1})$ and $A_{kj}=M_k\cap D_j$ for  $j\not\in \{k,k+1\}$. 
Then, using the law of total expectation we have
\begin{equation} \label{eq:SummaryThm1}
    \Delta V(\ell) = \sum_{k=1}^N \mathbb{E}(\Delta V(\ell)|A_k) \cdot  P(A_k) + 
    \sum_{k=1}^N\sum_{j \notin \{k,k+1\}} \mathbb{E}(\Delta V(\ell)|A_{kj}) \cdot
    P(A_{kj}).
\end{equation}

Since the arrivals of users and drivers are independent, we have $P(A_k)=\frac{(\ell_k+\ell_{k+1})}{2}(\ell_k+\ell_{k+1})$ and $P(A_{kj})=\frac{(\ell_k+\ell_{k+1})}{2}\ell_j$.
Hence, it remains to compute the conditional expectations appearing in this formula. To this end, let us suppose that we are in event $M_k$ with $I_k,I_{k+1}$ the two intervals that were merged, and let $I_M=I_k \cup I_{k+1}$ the merged interval, with respective lengths  $\ell_k,\ell_{k+1},\ell_k+\ell_{k+1}$. Consider the two cases.

\paragraph{Case 1: the new driver appears within the merged interval (event $A_k$).} Before we had two intervals of length $\ell_k,\ell_{k+1}$, and now they changed to $z,\ell_k+\ell_{k+1}-z$ for some $z\in[0,\ell_k+\ell_{k+1}]$. Since the remaining $\ell_j$'s are unchanged, the change in $V$ only depends on the difference these two terms, which is
\begin{equation} \label{Eq:NewDriverInsideMergedInterval1}
 DV_1(\ell_k,\ell_{k+1},z)=   z^2 + (\ell_k+\ell_{k+1}-z)^2 - [\ell_k^2 + \ell_{k+1}^2] = 2[z^2+\ell_k\ell_{k+1}-z(\ell_k+\ell_{k+1})].
\end{equation}
From this equation  we can see that:

\begin{itemize}
    \item If $\ell_k=0$ or $\ell_{k+1}=0$, the change is zero or negative. This is sensible, as $\ell_k=0$ means that two drivers were in the same place, and now they have been separated (or there was no change).
    \item If $z=\frac{\ell_k+\ell_{k+1}}{2}$, the change is zero or negative. This is sensible because the merged interval is now perfectly divided.
    \item If $\ell_k=\ell_{k+1}$, the change is zero or positive. The reason is the same as in the previous bullet point: the interval used to be perfectly divided.
\end{itemize}
Given that users arrivals are uniform on $[0,1]$, we have that $z$ is also uniform in $[0,\ell_k+\ell_{k+1}]$. Hence, the expected value of the change conditional on $M_k$ and  the new driver appearing inside $I_k\cup I_{k+1}$ is
\begin{equation}\label{Eq:SummaryChange1}
\mathbb{E}(\Delta V(\ell)|A_k)=  \frac{1}{\ell_k+\ell_{k+1}} \int_0^{\ell_k+\ell_{k+1}} 2[z^2+\ell_k\ell_{k+1}-z(\ell_k+\ell_{k+1})] dz = \frac{1}{3}[4\ell_k\ell_{k+1}-\ell_k^2-\ell_{k+1}^2].
\end{equation}

\paragraph{Case 2: the new driver appears outside the merged interval (events $A_{kj}$).}
Let us denote by $\ell_j$ the length of the interval divided by the new driver into two shorter intervals of length $z, \ell_j-z$ respectively. The change in $V$ is now 
\begin{equation}\label{Eq:NewDriverNOTInsideMergedInterval1}
    DV_2(\ell_k,\ell_{k+1},\ell_j,z)=(\ell_k+\ell_{k+1})^2+(l_j-z)^2+z^2 - [\ell_k^2+\ell_{k+1}^2+\ell_j^2] = 2[z^2-z\ell_j+\ell_k\ell_{k+1}].
\end{equation}
The analysis of Eq. \eqref{Eq:NewDriverNOTInsideMergedInterval1} reveals that:
\begin{itemize}
    \item If $\ell_j \approx 0$ (implying that $z\approx 0$ as well), the change is positive, which is sensible because we are dividing an interval that was already small.
    \item If $z\approx 0$, the change is positive, which happens because the new driver appeared very close to a previous one, contribution almost nothing to the coverage of the segment.
    \item If $\ell_k \approx 0$ or $\ell_{k+1} \approx 0$, the change is negative. This occurs because the there were two drivers very close to each other and one of them was replaced by someone in a different position, improving the overall situation.
\end{itemize}
The conditional expected change in this case is given by 
\begin{equation}\label{Eq:SummaryChange2}
    \mathbb{E}(\Delta V(\ell)|C_{kj})=\frac{1}{\ell_j}\int_0^{\ell_j} 2[z^2-z\ell_j+\ell_k\ell_{k+1}] dz= 2\left[\ell_k\ell_{k+1}-\frac{\ell_j}{6} \right]
\end{equation}

\noindent The result now follows directly by substituting \eqref{Eq:SummaryChange1} and \eqref{Eq:SummaryChange2} and the corresponding probabilities in \eqref{eq:SummaryThm1}.
\end{proof}

\subsection*{Proof of Theorem \ref{thm:IGDWorseThanUniform}}
Let us start analysing (*1) from Eq. \eqref{ChangeInV}. To do this, we will show that actually the expected value of Eq. \eqref{Eq:SummaryChange1} is equal to zero, which suffices because then the first term in Eq. \eqref{eq:SummaryThm1} nullifies. We omite the multiplying $1/3$ as it does not affect the zero result.
$$E_\ell([4\ell_k\ell_{k+1}-\ell_k^2-\ell_{k+1}^2])=\int_0^1 E([4\ell_k\ell_{k+1}-\ell_k^2-\ell_{k+1}^2]|\ell_k+\ell_{k+1}=L)\cdot f_{\ell_k+\ell_{k+1}}(L)dL$$
where $f_X$ stands for the pdf of $X$. Note that this proof assumes the continuous case, while the discrete case is analogous replacing the integral by a sum. We now show that every term in this integral is zero. In fact,
\begin{align*}
  E([4\ell_k\ell_{k+1}-\ell_k^2-\ell_{k+1}^2]|\ell_k+\ell_{k+1}=L)&=E([4\ell_k(L-\ell_k)-\ell_k^2-(L-\ell_k)^2]|\ell_k+\ell_{k+1}=L) \\
  &=6LE(\ell_k|\ell_k+\ell_{k+1}=L)-L^2-6E(\ell_k^2|\ell_k+\ell_{k+1}=L).
\end{align*}
Given $\ell_k+\ell_{k+1}=L$, we have that $\ell_k$ is uniformly distributed in $[0,L]$. Therefore, $E(\ell_k|\ell_k+\ell_{k+1}=L)=L/2$, and $E(\ell_k^2|\ell_k+\ell_{k+1}=L)=L^2/3$, so putting everything together, the term inside the integral becomes
$$\mbox{$6\frac{L^2}{2}-L^2-6\frac{L^2}{3}=0$}.$$

We now show (*2) is strictly positive. Note that we can shift the vector without affecting the result, so we assume that the first driver is located at zero. We then have $n-1$ uniform variables in $[0,1]$, so the interval vector $(\ell_1,\ldots,\ell_N)$ follows a Dirichlet distribution\footnote{We use here well-known facts about the distribution of the intervals that follows a random partition of the unit interval. The original source is \citep{renyi1953theory}, but there are numerous lecture notes and links that describe them as well.}  with parameter $(1,\ldots,1)$, where the vector has length $N$. In that case, it is also an established fact that
$$E(\mbox{$\prod$}_{i=1}^N\ell_i^{\beta_i})=\frac{\Gamma(N)}{\Gamma(\sum_i(1+\beta_i))}\cdot\mbox{$\prod$}_{i=1}^N \Gamma(1+\beta_i).$$
Note that in the expression above $\beta$ is general, but for integers, we have that $\Gamma(k)=(k-1)!$. So in our case we have to compute
$E(\ell_k^2\ell_{k+1}\ell_j-\ell_k\ell_j^3/6)$.
Using the general formula above, we get
$$
\mbox{$E(\ell_k^2\ell_{k+1}\ell_j)=\frac{\Gamma(N)}{(N+3)!}\cdot 2!$}$$
$$
\mbox{$E(\ell_k\ell_j^3/6)=\frac{\Gamma(N)}{(N+3)!}\cdot 3! \cdot \frac{1}{6}.$}
$$
As $\frac{3!}{6}=1<2!$, we conclude that $E(\ell_k^2\ell_{k+1}\ell_j-\ell_k\ell_j^3/6)>0$, i.e., $(*2)$ is positive completing the proof.
\qed

\end{document}